\documentclass[12pt]{amsart}
\usepackage{amssymb,mathrsfs,longtable}
\usepackage[matrix,arrow,curve]{xy}
\usepackage{setspace}
\usepackage{multirow}
\usepackage{epsfig,color}
\usepackage{textcomp}
\usepackage{url}

\newcounter{cequation}[section]

\newtheorem{theorem}[cequation]{Theorem}
\newtheorem*{theorem*}{Theorem}
\newtheorem{lemma}[cequation]{Lemma}
\newtheorem{corollary}[cequation]{Corollary}

\newtheorem{proposition}[cequation]{Proposition}

\theoremstyle{definition}
\newtheorem{example}[cequation]{Example}
\newtheorem{definition}[cequation]{Definition}
\newtheorem*{definition*}{Definition}

\theoremstyle{remark}
\newtheorem{remark}[cequation]{Remark}

\renewcommand\arraystretch{1.3}

\makeatletter\@addtoreset{equation}{section}
\makeatletter\@addtoreset{section}{part}

\makeatother

\def \O {\mathcal{O}}

\def \CC {\mathbb{C}}

\def \P {\mathbb{P}}
\def \PP {\mathbb{P}}
\def \R {\mathbb{R}}

\def \Z {\mathbb{Z}}
\def \ZZ {\mathbb{Z}}

\def \ee {\mathbf{e}}
\def \Aff {\mathbb{A}}
\def \DD {\mathbb{D}}

\def \Pic {\mathrm{Pic}\,}

\def \bideg {\mathrm{bideg}}
\newcommand{\Spec}{\mathrm{Spec}\,}

\def \ge {\geqslant}
\def \le {\leqslant}

\usepackage[dvipsnames]{xcolor}
\usepackage{graphicx}

\title{Bounds for smooth Fano weighted complete intersections}

\author{Victor Przyjalkowski and Constantin Shramov}

\address{\emph{Victor Przyjalkowski}
\newline
\textnormal{Steklov Mathematical Institute of RAS, 8 Gubkina street, Moscow 119991, Russia.
}
\newline
\textnormal{National Research University Higher School of Economics, Russian Federation,
Laboratory of Mirror Symmetry, NRU HSE, 6 Usacheva str., Moscow, Russia, 119048.
}
\newline
\textnormal{\texttt{victorprz@mi.ras.ru, victorprz@gmail.com}}}

\address{\emph{Constantin Shramov}
\newline
\textnormal{Steklov Mathematical Institute of RAS,
8 Gubkina street, Moscow 119991, Russia.
}
\newline
\textnormal{National Research University Higher School of Economics, Laboratory of Algebraic Geometry, 6 Usacheva str., Moscow, 119048, Russia.
}
\newline
\textnormal{\texttt{costya.shramov@gmail.com}}}

\thanks{Victor Przyjalkowski was partially supported by Laboratory of Mirror Symmetry NRU HSE, RF Government grant, ag. \textnumero~14.641.31.0001.
Constantin Shramov was supported by the Russian Academic Excellence Project ``5-100'',
by the Program of the Presidium of
the Russian Academy of Sciences \textnumero~01 ``Fundamental Mathematics and
its Applications'' under grant PRAS-18-01,
and by the Foundation for the
Advancement of Theoretical Physics and Mathematics ``BASIS''.
Both authors are Young Russian Mathematics award winners and would like to thank
its sponsors and jury.
}

\begin{document}

\begin{abstract}

We prove that if a smooth variety with non-positive canonical class can be embedded into a weighted projective space of dimension~$n$ as a well
formed complete intersection
and it is not an intersection with a linear cone therein, then the weights of the weighted projective space do not exceed~\mbox{$n+1$}.
Based on this bound we classify all smooth Fano complete intersections of dimensions~$4$ and~$5$,
and compute their invariants.

\end{abstract}

\maketitle

\section{Introduction}
\label{section:intro}

Fano varieties are one of the important classes of algebraic varieties,
both from birational and biregular points of view.
It is known that smooth Fano varieties of a given dimension are bounded,
see~\cite[Theorem~0.2]{KollarMiyaokaMori}, so that one can hope for
their explicit classification (actually this is known also for $\varepsilon$-log terminal Fano varieties, see~\cite{Bi16},
but in this case any kind of explicit classification is hardly possible even in dimension $3$).
The only smooth Fano curve is~$\P^1$.
Smooth Fano varieties of dimension~$2$ are known as del Pezzo surfaces, and
they were classified long ago.
Smooth Fano threefolds were classified by V.\,Iskovskikh (see~\cite{Is77},~\cite{Is78}, or
\cite[\S12]{IP99}), and S.\,Mori and S.\,Mukai (see, \cite{MM82} and~\cite{Mu88}).
The most important and hard part of this classification concerns Fano varieties with Picard rank~$1$.
In dimension $3$ such varieties (at least if they are general in the corresponding deformation family)
appear to be either complete intersections
in a weighted projective space,
or zero loci of sections of homogeneous vector bundles on Grassmannians.
In dimensions $4$ and higher no complete classification is known, and
at the moment no reasonable approach to the problem is yet in sight.
Still there are partial classification results, including
the list of all smooth Fano fourfolds of index at least~$2$ or Picard rank greater than one and Fano varieties of high coindex
(see~\cite{Fu80},~\cite{Fu81},~\cite{Fu84},~\cite{Mu84},~\cite{Wi87},~\cite{Mu89},~\cite{Wi90}),
smooth Fano fourfolds that are zero loci of sections of homogeneous vector bundles on Grassmannians
(see~\cite{Kuchle-Grass}, \cite[\S4]{Kuchle-Geography}, \cite{Kuznetsov-Kuchle5Folds}), smooth Fano fourfolds that are
weighted complete intersections (see~\cite[Proposition~2.2.1]{Kuchle-Geography}), and some
other sporadic results (see~\cite[\S3]{Kuchle-Geography}). The purpose of this paper
is to study smooth Fano weighted complete intersections, and give effective numerical bounds that allow
to classify them.

To be able to classify weighted complete intersections of a given dimension satisfying some
nice properties, one needs an effective bound on the corresponding discrete parameters.
In~\cite[Theorem~1.3]{ChenChenChen} such a bound was obtained for codimension of a quasi-smooth
(see Definition~\ref{definition:quasi-smooth} below)
weighted complete intersection.
In~\cite[Theorem~1.3]{Chen} the degrees were bounded in terms of canonical volume and discrepancies.

We will be interested in the case of smooth Fano varieties that can be described as complete intersections in weighted projective spaces.
Also, we will deal with the case when our weighted complete intersection is Fano or Calabi--Yau.
In both cases by adjunction formula it is
actually enough to bound the weights of the corresponding weighted projective space.

Let $\P=\P(a_0,\ldots,a_n)$ be a weighted projective space, and $X\subset\P$ be a weighted complete intersection of multidegree $(d_1,\ldots,d_c)$ for some $c\ge 0$. We will usually assume that $X$
\emph{is not an intersection with a linear cone}, i.e. one has $d_j\neq a_i$ for all $i$ and $j$, cf. Remark~\ref{remark:linear}.
Finally, it is convenient to assume that $X$ is \emph{well formed}, see Definition~\ref{definition:well-formed-WCI}
and Theorem~\ref{theorem:well-formed} below.

The main result of this paper is the following.

\begin{theorem}
\label{theorem:main}
Let $X\subset\P(a_0,\ldots,a_n)$, $n\ge 2$, be a smooth well formed weighted complete intersection
of multidegree $(d_1,\ldots,d_c)$. Suppose that $X$ is not an intersection with a linear cone.
If $X$ is Fano, then for every $0\le i\le n$ one has $a_i\le n$, and for every~\mbox{$1\le j\le c$} one has~\mbox{$d_j\le n(n+1)$}.
Similarly, if $X$ is Calabi--Yau, then for every $0\le i\le n$ one has~\mbox{$a_i\le n+1$},
and for every $1\le j\le c$ one has
$d_j\le (n+1)^2$.
\end{theorem}

To prove Theorem~\ref{theorem:main} we use the following approach. Exploiting smoothness assumption,
we write down a bunch of necessary conditions on the parameters~$a_i$ and~$d_j$, which appear to be
inequalities (sometimes involving products of weights or degrees).
On the other hand, Fano or Calabi--Yau condition implies an inequality between the sums
of~$a_i$ and~$d_j$. Then we treat all these
inequalities as if $a_i$ and $d_j$ were arbitrary real numbers,
and solve the corresponding optimization problem
using the standard down-to-earth method of Lagrange multipliers.

The bounds for the largest weight of $\P$ given by Theorem~\ref{theorem:main} are sharp for an infinite set of dimensions, see Remark~\ref{remark:sharp} below.

Using Theorem~\ref{theorem:main}, we will give a classification
of weighted Fano complete intersections of dimensions $4$ and $5$, see \S\ref{section:tables} below.
Note that a classification of four-dimensional weighted Fano complete intersections
was already obtained by O.\,K\"uchle in \cite[Proposition~2.2.1]{Kuchle-Geography}; his method builds on a classification of weighted homogeneous polarized Calabi--Yau complete intersections, see~\cite[Main Theorem~II]{Og91}.
In a way this is closer to the methods that were classically used in a classification of Fano threefolds,
but one can hardly expect that it can be easily generalized to higher dimensions.

The plan of the paper is as follows.
In~\S\ref{section:smooth}
we recall some basic properties of weighted complete intersections. In~\S\ref{section:bound} we prove Theorem~\ref{theorem:main}.
In~\S\ref{section:Hodge} we explain the (well known) method that can be used
to compute Hodge numbers of smooth weighted complete intersections.
In~\S\ref{section:tables} we provide a classification of smooth Fano weighted complete intersections of dimensions $4$ and $5$.
Finally, in Appendix~\ref{section:optimization} we collect some (nearly elementary) auxiliary material
used in~\S\ref{section:bound}.

\medskip

\textbf{Notation and conventions.}
All varieties are compact and are defined over the field of complex numbers $\CC$.
For a bigraded ring $R$ we denote its $(p,q)$-component by $R_{(p,q)}$.
For a weighted complete intersection $X$ in $\PP(a_0,\ldots,a_n)$ of multidegree $(d_1,\ldots,d_c)$
the number $\sum a_i-\sum d_j$ is denoted by $I(X)$.

\medskip

\textbf{Acknowledgments.}
We are grateful to K.\,Besov, A.\,Corti, A.\,Kuznetsov, T.\,Okada, and Yu.\,Prokhorov for useful discussions,
and to A.\,Harder who helped us to compute Hodge numbers in~\S\ref{section:tables}. Special thanks go to the referees for numerous valuable comments.

\section{Smoothness}
\label{section:smooth}

We recall here some basic properties of weighted complete intersections. We refer the reader
to~\cite{Do82} and~\cite{IF00} for more details.
Let $a_0,\ldots,a_n$ be positive integers. Consider the graded algebra~\mbox{$\CC[x_0,\ldots,x_n]$},
where the grading is defined by assigning the weights $a_i$ to the variables~$x_i$.
Put
$$
\P=\P(a_0,\ldots,a_n)=\mathrm{Proj}\,\CC[x_0,\ldots,x_n].
$$

\begin{definition}[{see~\cite[Definition 5.11]{IF00}}]
\label{definition:well-formed-P}
The weighted projective space $\P$ is said to be \emph{well formed} if the greatest common divisor of any $n$ of the weights~$a_i$ is~$1$.
\end{definition}

Any weighted projective space is isomorphic to a well formed one, see~\cite[1.3.1]{Do82}.

\begin{lemma}[{see~\cite[5.15]{IF00}}]
\label{lemma:singularities-of-P}
Suppose that $\PP$ is well formed. Then the singular locus of~$\PP$ is a union of strata
$$
\Lambda_J=\left\{(x_0:\ldots:x_n) \mid x_j=0 \text{\ for all\ } j\notin J\right\}
$$
for all subsets $J\subset \{0,\ldots,n\}$ such that the greatest common divisor of the weights~$a_j$
for~\mbox{$j\in J$} is greater than~$1$.
\end{lemma}

\begin{definition}[{see~\cite[Definition 6.9]{IF00}}]
\label{definition:well-formed-WCI}
A subvariety $X\subset \PP$ of codimension $c$ is said to be \emph{well formed}
if~$\PP$ is well formed and
$$
\mathrm{codim}_X \left( X\cap\mathrm{Sing}\,\P \right)\ge 2.
$$
\end{definition}

The following notion is a replacement of smoothness suitable for subvarieties of
weighted projective spaces.

\begin{definition}[{see~\cite[Definition 6.3]{IF00}}]
\label{definition:quasi-smooth}
Let $p\colon \Aff^{n+1}\setminus\{0\}\to \PP$ be the natural projection. A subvariety $X\subset \PP$ is said to be \emph{quasi-smooth} if $p^{-1}(X)$ is smooth.
\end{definition}

We say that a variety $X\subset\P$ of codimension $c$ is a \emph{weighted complete
intersection of multidegree $(d_1,\ldots,d_c)$} if its weighted homogeneous ideal in $\CC[x_0,\ldots,x_n]$
is generated by a regular sequence of $c$ homogeneous elements of degrees $d_1,\ldots,d_c$.
Note that in general a weighted complete intersection
is not even locally a complete intersection in the usual sense.

\begin{remark}\label{remark:WCI-warning}
It is possible that a weighted complete intersection
of a given multidegree in~$\P$ does not exist, even if $c$ is small.
For example, there is no such thing as a hypersurface of degree~\mbox{$d<\min(a_0,\ldots,a_n)$} in~$\P$,
or a weighted complete intersection of multidegree $(2,2)$ in~\mbox{$\P(1,3,4,5)$}.
\end{remark}

Singularities of quasi-smooth well formed weighted complete intersections
can be easily described.

\begin{proposition}[{see \cite[Proposition 8]{Di86}}]
\label{proposition:singularities-of-X}
Let $X\subset \PP$ be a quasi-smooth well formed weighted complete intersection. Then the singular locus of $X$ is
the intersection of $X$ with the singular locus of $\PP$.
\end{proposition}

\begin{remark}
Note that the definition of ``general position'' in~\cite{Di86}
coincides with our definition of well formedness.
\end{remark}

Recall that the weighted complete intersection $X$ is said to be an intersection
with a linear cone if one has $d_j=a_i$ for some~$i$ and~$j$.

\begin{remark}
\label{remark:linear}
If this condition fails, one can exclude the $i$-th weighted homogeneous coordinate and think about
$X$ as a weighted complete intersection in a weighted projective space of lower dimension, provided that
$X$ is general enough, cf. Remark~\ref{remark:linear-cone} below. Note however that
in general this new weighted projective space may fail to be well formed, and the new
weighted complete intersection may fail to be nice in other ways as well.
\end{remark}

It appears that the assumptions that the complete intersection is well formed, quasi-smooth, and
is not an intersection with a linear cone are not always independent.
In principle it can allow us to drop some of the assumptions in the rest of the paper,
but we will refrain from doing so to keep the assertions more explicit.

\begin{theorem}[{see \cite[Theorem 6.17]{IF00}}]
\label{theorem:well-formed}
Suppose that the weighted projective space $\PP$ is well formed. Then
any quasi-smooth complete intersection of dimension at least~$3$ in~$\PP$
is either an intersection with a linear cone
or well formed.
\end{theorem}

There is the following version of the adjunction
formula that holds for quasi-smooth well formed weighted complete intersections.

\begin{lemma}[{see~\cite[Theorem 3.3.4]{Do82}, \cite[6.14]{IF00}}]
\label{lemma:adjunction}
Let $X\subset \PP$ be a quasi-smooth well formed weighted complete intersection of multidegree $(d_1,\ldots,d_c)$.
Then
$$
\omega_X=\O_X\left(\sum d_i-\sum a_j\right).
$$
\end{lemma}

Along with quasi-smooth weighted complete intersections one may consider
those weighted complete intersections that are smooth in the usual sense.
Note that if we do not assume that a weighted complete intersection $X\subset\PP$ is well formed,
then $X$ may be smooth even if it passes through the
singularities of $\PP$. An example of such a behavior is given by a line
(which is a hypersurface of degree~$1$) on the two-dimensional quadratic cone~\mbox{$\PP=\PP(1,1,2)$}.
On the other hand, if $X$ is both smooth and well formed, then it must be disjoint from
the singular locus of~$\PP$. Furthermore, in this case $X$ can be shown to be quasi-smooth
(although this fact is not quite obvious, see Corollary~\ref{corollary:smooth-vs-quasismooth} below).

\begin{proposition}\label{proposition:smooth-does-not-pass}
Let $X\subset\PP$ be a smooth well formed weighted complete intersection.
Then $X$ does not pass through singular points of $\PP$.
\end{proposition}
\begin{proof}
Suppose that $X$ contains a singular point $P$ of $\PP$.
Let $U\subset\PP$ be an affine neighborhood of $P$,
and $\pi\colon\tilde{U}\to U$ be its natural
finite cover (see \cite[5.3]{IF00}), so that $\tilde{U}$ is isomorphic to
an open subset of~$\mathbb{A}^n$, and $\pi$ is a quotient
by a group $\ZZ/r\ZZ$ for some~\mbox{$r>1$}. Put $V=X\cap U$.
Let $\Sigma$ be the singular locus of $U$. Since $X$ is well formed,
the intersection of $\Sigma$ with $V$ has codimension at least $2$ in $V$.
Let $\tilde{V}$ be the preimage of $V$ with respect to~$\pi$, and
let $\pi_V\colon \tilde{V}\to V$ be the corresponding finite cover.
Then $\tilde{V}$ is a complete
intersection in~$U$.
Note also that $\pi_V$ is \'etale outside of $\Sigma$, and thus
$\tilde{V}$ is smooth in codimension~$1$.
In particular, $\tilde{V}$ is Cohen--Macaulay, see~\cite[\S18.5]{Eisenbud}.

Since $V$ is smooth by assumption, we can choose a simply connected analytic neighborhood $\mathcal{V}$ of $P$ in $V$.
Let $\tilde{\mathcal{V}}$
be its preimage with respect to $\pi_V$.
Put $\mathcal{V}^o=\mathcal{V}\setminus\Sigma$.
Since the intersection of $\Sigma$ with $V$ has codimension at least $2$ in $V$,
the open set $\mathcal{V}^o$ is also simply connected.
Let $\tilde{\mathcal{V}}^o$ be the preimage of
$\mathcal{V}^o$ with respect to $\pi_V$.
The morphism~$\pi_V$ is \'etale
over $\mathcal{V}^o$,
so that the complex space~$\tilde{\mathcal{V}}^o$ splits into a union
of its connected components~\mbox{$\tilde{\mathcal{V}}^o_1,\ldots,\tilde{\mathcal{V}}^o_r$, $r\ge 2$}.
Let $\tilde{\mathcal{V}}_1,\ldots,\tilde{\mathcal{V}}_r$
be the closures of $\tilde{\mathcal{V}}^o_1,\ldots,\tilde{\mathcal{V}}^o_r$ in $\tilde{\mathcal{V}}$.
Then~\mbox{$\tilde{\mathcal{V}}_1,\ldots,\tilde{\mathcal{V}}_r$} intersect each other
(in particular) at the point $\pi^{-1}(P)$,
so that~$\tilde{\mathcal{V}}$ is connected.
Since $\tilde{V}$ is Cohen--Macaulay, it follows from~\cite[Theorem~18.12]{Eisenbud}
that there is an index $2\le k\le r$ such that
the intersection $\mathcal{Z}=\tilde{\mathcal{V}}_1\cap\tilde{\mathcal{V}}_k$
has codimension $1$ in $\tilde{\mathcal{V}}_1$.
This means that the variety~$\tilde{\mathcal{V}}$, and thus also~$\tilde{V}$, is singular at the points of~$\mathcal{Z}$.
Since $\pi_V$ is \'etale at a general point of $\mathcal{Z}$, we conclude that
$V$ is singular at a general point of $\pi_V(\mathcal{Z})$, which is a contradiction.
\end{proof}

\begin{remark}
A.\,Kuznetsov pointed out that there is an alternative proof of
Proposition~\ref{proposition:smooth-does-not-pass} that is purely algebraic
and does not depend on the base field. Namely, in the notation of
our proof of Proposition~\ref{proposition:smooth-does-not-pass}
the variety $\tilde{V}$ is normal (since it is a locally complete intersection
smooth in codimension~$1$). Since $V$ is smooth by assumption,
the branch locus $R$ of $\pi_V$ has codimension $1$ in $V$ by the purity of the branch locus,
see~\mbox{\cite[Theorem~1.4]{Auslander}}.
Now we can obtain a contradiction as above.
Still we prefer to keep the original proof of Proposition~\ref{proposition:smooth-does-not-pass},
since we believe that it makes the geometric reason explaining why this property
holds more transparent, and our base field is~$\CC$ anyway.
\end{remark}

\begin{remark}
The assertion of Proposition~\ref{proposition:smooth-does-not-pass}
fails without the assumption that $X$ is well formed. Indeed, suppose that $a_0=a_1=1$
and $a_2=\ldots=a_n=2$; put $c=1$ and~\mbox{$d_1=2$}. Then $X$ is not well formed, but
it is smooth since it is isomorphic to~$\P^{n-1}$. However, it passes through
singular points of~$\P$. For $n=2$ this example gives a line on a usual quadratic cone.
\end{remark}

As an application of Proposition~\ref{proposition:smooth-does-not-pass}
one can show that being smooth is a stronger condition than being quasi-smooth,
provided that we work with well formed weighted complete intersections.

\begin{corollary}\label{corollary:smooth-vs-quasismooth}
Let $X\subset\PP$ be a smooth well formed weighted complete intersection.
Then $X$ is quasi-smooth.
\end{corollary}
\begin{proof}
The morphism $p\colon \Aff^{n+1}\setminus\{0\}\to \PP$
is a locally trivial $\CC^*$-bundle over the non-singular part $U$ of $\PP$,
while $X$ is contained in $U$ by Proposition~\ref{proposition:smooth-does-not-pass}.
Hence $p$ is a locally trivial $\CC^*$-bundle over $X$, and thus the preimage of $X$
with respect to $p$ is smooth, which means that $X$ itself is quasi-smooth.
\end{proof}

If the weighted projective space $\PP$ is not well formed,
then the assertion of Corollary~\ref{corollary:smooth-vs-quasismooth}
may fail. Thus, the curve $X$ in $\PP=\PP(1,2,2)$
given by equation
$$
x_0^2x_1+x_2^2=0
$$
is not quasi-smooth because the cone over this curve
in $\mathbb{A}^3$ is singular, for instance, at the point $(0,1,0)$.
On the other hand, $X$ is
isomorphic to a conic given by the equation
$$
z_0z_1+z_2^2=0
$$
in $\PP^2\cong \PP(1,2,2)$, and thus it is smooth.
However, we do not know whether there exists an example of a
smooth (but not well formed)
weighted complete intersection $X$ in a well formed
weighted projective space~$\PP$ such that $X$ is not quasi-smooth.

Another consequence of Proposition~\ref{proposition:smooth-does-not-pass}
is the following result.

\begin{lemma}[{cf.~\cite[Proposition~4.1]{ChenChenChen}}]
\label{lemma:smooth-condition}
Let $X\subset\PP$ be a smooth well formed weighted complete intersection
of multidegree $(d_1,\ldots,d_c)$.
Then for every $k$ and every choice of~$k$ weights~\mbox{$a_{i_1},\ldots,a_{i_k}$, $i_1<\ldots<i_k$},
such that their greatest common divisor $\delta$ is greater than~$1$
there exist $k$ degrees $d_{s_1},\ldots,d_{s_k}$, $s_1<\ldots<s_k$,
such that their greatest common divisor is divisible by~$\delta$.
\end{lemma}
\begin{proof}
Choose a positive integer $k$, and suppose that there are $k$
weights $a_{i_1},\ldots,a_{i_k}$ with~\mbox{$i_1<\ldots<i_k$},
such that their greatest common divisor $\delta$ is greater than $1$. Let $t$
be the number of degrees~$d_j$ that are divisible by $\delta$. Suppose
that $t<k$. We claim that in this case $X$ is singular. Indeed,
let $f_1=\ldots=f_c=0$ be the equations of $X$ in $\P$, so that~\mbox{$\deg(f_j)=d_j$}. Let $J$ be the set of indices $j$ such that $d_j$ is divisible by $\delta$,
and let $\Lambda$ be the subvariety in $\P$ given by equations $x_j=0$ for $j\in J$. Then for any~\mbox{$j'\not\in J$} the polynomial~$f_{j'}$ does not contain monomials that depend only on
$x_j$ with~\mbox{$j\in J$}. On the other hand, the equations $f_j=0$ for $j\in J$ cut out a non-empty
subset of~$\Lambda$. Since~\mbox{$\delta>1$}, the weighted projective space
$\P$ is singular along $\Lambda$, see Lemma~\ref{lemma:singularities-of-P}.
This gives a contradiction with Proposition~\ref{proposition:smooth-does-not-pass}.
\end{proof}

\begin{remark}
The condition provided by Lemma~\ref{lemma:smooth-condition} is only necessary for
the weighted complete intersection $X$ to be smooth, but not sufficient. For example, assume
that~\mbox{$a_0=\ldots=a_r=1$}, while~\mbox{$2<a_{r+1}\le\ldots\le a_n$}, and
$a_j$ are pairwise coprime. Choose~$d_1$ and~$d_2$
so that~$d_2$ is divisible by all $a_j$, and~\mbox{$2\le d_1<a_{r+1}$}.
Then a general weighted complete intersection $X$ of multidegree~\mbox{$(d_1,d_2)$} in
$\P$ is not smooth provided that $r\le n-2$; moreover, it is reducible if~\mbox{$r=1$},
and non-reduced if~\mbox{$r=0$}. Another way how smoothness may fail is illustrated by
an example of a weighted complete intersection $X$ of multidegree~\mbox{$(2,30)$} in
$\P$ when $a_0=\ldots=a_{n-2}=1$, $a_{n-1}=6$, and $a_n=10$; in this case we see
that the assertion of Proposition~\ref{proposition:smooth-does-not-pass} does not hold,
so that $X$ is singular. See also Lemma~\ref{lemma:d-2a}(i) below.
\end{remark}

\section{Weight bound}
\label{section:bound}

In this section we derive Theorem~\ref{theorem:main} from elementary
results of Appendix~\ref{section:optimization}. The method we use here is somewhat
similar to~\cite[\S3]{CheltsovShramov-DelPezzoZoo}.

Let $X\subset\P=\P(a_0,\ldots,a_n)$, $n\ge 2$, be a smooth well formed weighted complete intersection
of multidegree $(d_1,\ldots,d_c)$ that is not an intersection with a linear cone.
We can assume that $X$ is \emph{normalized}, i.e. that inequalities
$a_0\le\ldots\le a_n$ and $d_1\le\ldots\le d_c$ hold.
Moreover,  if $c=0$, then one has $X\cong\P\cong\P^n$, and there is nothing to prove; therefore, we will always assume that $c\ge 1$.

We need an auxiliary result that is easy to establish and well known to experts.

\begin{lemma}
\label{lemma:d-2a}
Let $X\subset\P=\P(a_0,\ldots,a_n)$, $n\ge 2$, be a smooth well formed normalized weighted complete intersection of multidegree $(d_1,\ldots,d_c)$ that is not an intersection with a linear cone.
Then the following assertions hold.
\begin{itemize}
\item[(i)]
One has $d_{c-k}> a_{n-k}$ for all $1\le k\le c-1$.

\item[(ii)]
One has $d_c\ge 2a_n$.

\item[(iii)]
The integer $a_0\cdot\ldots\cdot a_n$ divides the integer $d_1\cdot\ldots\cdot d_c$.
\end{itemize}
\end{lemma}
\begin{proof}
Assertion~(i) is given by \cite[Lemma~18.14]{IF00} and holds under a weaker assumption
of quasi-smoothness.
If $a_n=1$, then the remaining assertions of the lemma obviously hold, and thus we can
assume that $a_n>1$.

Let $x_0,\ldots,x_n$ be homogeneous coordinates on $\P$ of weights
$a_0,\ldots,a_n$, respectively. Let~\mbox{$f_1=\ldots=f_c=0$} be the equations of $X$ in $\P$, so that~\mbox{$\deg(f_j)=d_j$}.

Suppose that $d_c<2a_n$.
Then none of $f_j$ contains a monomial $x_n^r$ with non-zero coefficient
if $r\ge 2$. Also, since $X$ is not an intersection with a linear cone, none
of $f_j$ contains a monomial $x_n$ with non-zero coefficient
either. Therefore, we see that every $f_j$ vanishes at the point $P$
given by $x_0=\ldots=x_{n-1}=0$, so that $X$ passes through $P$. On the other hand,
$P$ is a singular point of $\P$ by Lemma~\ref{lemma:singularities-of-P} because $a_n>1$.
Thus Proposition~\ref{proposition:singularities-of-X} implies that $X$ is singular at $P$,
which is a contradiction.
This gives assertion~(ii).

To prove assertion~(iii), choose a prime number $p$, and denote by $\nu_p^{(r)}(a_0,\ldots,a_n)$
the number of the weights~$a_i$ that are divisible by $p^r$.
Similarly, denote by $\nu_p^{(r)}(d_1,\ldots,d_c)$
the number of the degrees~$d_j$ that are divisible by $p^r$.
By Lemma~\ref{lemma:smooth-condition} for every $r$ one has
$$
\nu_p^{(r)}(a_0,\ldots,a_n)\le \nu_p^{(r)}(d_1,\ldots,d_c).
$$
This implies that the $p$-adic valuation of the integer
$a_0\cdot\ldots\cdot a_n$ does not exceed the $p$-adic valuation of the integer
$d_1\cdot\ldots\cdot d_c$.
Since this holds for an arbitrary prime $p$, we obtain assertion~(iii).
\end{proof}

Now we are ready to prove Theorem~\ref{theorem:main}.
Recall that it asserts the bounds $a_i\le n$ and~\mbox{$d_j\le n(n+1)$}
if $X$ is Fano, and the bounds~\mbox{$a_i\le n+1$},
and $d_j\le (n+1)^2$ if $X$ is Calabi--Yau.

\begin{proof}[Proof of Theorem~\ref{theorem:main}]
Put $N=n+1$. Denote $A_{i+1}=a_{n-i}$ for $0\le i\le n$,
and~\mbox{$D_j=d_{c-j+1}$} for $1\le j\le c$. Then one has $A_1\ge\ldots\ge A_N$ and $D_1\ge\ldots\ge D_c$. Moreover, by Lemma~\ref{lemma:d-2a}(i) one has $D_2> A_2, \ldots, D_c> A_c$.
By Lemma~\ref{lemma:d-2a}(ii) we also have $D_1\ge 2A_1$, and by Lemma~\ref{lemma:d-2a}(iii)
we have
$$
A_1\cdot\ldots\cdot A_N\le D_1\cdot\ldots\cdot D_c.
$$

Put
$$
L=a_0+\ldots+a_n-d_1-\ldots-d_c=A_1+\ldots+A_N-D_1-\ldots-D_c.
$$
Then $L>0$ provided that $X$ is Fano, and $L\ge 0$ provided that $X$ is Calabi--Yau.
This follows from Lemma~\ref{lemma:adjunction}.

Suppose that $X$ is a Fano variety. Then $N\ge 2c+1$ by~\cite[Theorem~1.3]{ChenChenChen}.
Therefore, Proposition~\ref{proposition:Lagrange} implies that
$A_1\le N-1$, which can be rewritten as $a_n\le n$.

Now suppose that $X$ is Calabi--Yau. Then a general weighted complete intersection of multidegree
$d_1,\ldots,d_c$ in $\P(1,a_0,\ldots,a_n)$ is a smooth well formed Fano weighted complete intersection
that is not an intersection with a linear cone. Thus one has~\mbox{$a_n\le n+1$}.

Since $X$ is normalized, we obtain similar inequalities for all other weights~$a_i$.
Finally, the inequalities for the degrees $d_j$ follow from the fact that
$L$ is positive if $X$ is Fano, and non-negative if $X$ is Calabi--Yau.
This completes the proof of Theorem~\ref{theorem:main}.
\end{proof}

\begin{remark}\label{remark:sharp}
The bounds for $a_n$ given by Theorem~\ref{theorem:main} are sharp for an infinite set of dimensions.
Indeed, if $n$ is odd, $a_0=\ldots=a_{n-2}=1$, $a_{n-1}=2$, and $a_n=n$, then a general hypersurface of weighted degree $2n$ in $\P$
is a smooth well formed Fano weighted complete intersection. Similarly, if $n$ is even, $a_0=\ldots=a_{n-2}=1$, $a_{n-1}=2$, and~\mbox{$a_n=n+1$},
then a general hypersurface of weighted degree $2n+2$ in $\P$
is a smooth well formed Calabi--Yau weighted complete intersection. However, we do not know if the bound for $a_n$ is attained for even $n$ in the case of Fano weighted
complete intersections, and if it is attained for odd $n$ in the case of Calabi--Yau weighted
complete intersections.
\end{remark}

Although Remark~\ref{remark:sharp} shows that the bound for the maximal weight~$a_n$ given by Theorem~\ref{theorem:main}
is more or less sharp, there are stronger bounds for some other weights~$a_i$ in certain cases.

\begin{lemma}\label{lemma:middle-bound}
Let $X\subset\P=\P(a_0,\ldots,a_n)$, $n\ge 2$, be a smooth well formed normalized weighted complete intersection of multidegree $(d_1,\ldots,d_c)$ that is not an intersection with a linear cone.
Suppose that $X$ is Fano or Calabi--Yau. Then for every $0\le k\le \dim X$
one has
$$
a_k< 2^{\frac{\dim X+1}{\dim X-k+1}}.
$$
Moreover, if $\dim X\ge 2$, then one has $a_0=a_1=1$.
\end{lemma}

\begin{proof}
For the first assertion, we mostly follow the proof of \cite[(2.6)]{ChenChenChen}.
Using Lemmas~\ref{lemma:adjunction}
and~\ref{lemma:d-2a}(i), one
gets
\begin{multline}\label{eq:Chens}
(\dim X+1)a_{\dim X+1}\ge(\dim X+1)a_{\dim X+1}-I(X)\ge\\
\ge a_0+\ldots+a_{\dim X}-I(X)=\sum_{j=1}^c\left(\frac{d_j}{a_{\dim X+j}}-1\right)a_{\dim X+j}\ge\\
\ge \left(\sum_{j=1}^c\frac{d_j}{a_{\dim X+j}}-c\right) a_{\dim X+1},
\end{multline}
so that
$$
\dim X+c+1\ge \sum_{j=1}^c\frac{d_j}{a_{\dim X+j}}.
$$
Thus
\begin{multline*}
\left(1+\frac{\dim X+1}{c}\right)^c\ge\left(\frac{\sum_{j=1}^c\frac{d_j}{a_{\dim X+j}}}{c}\right)^c\ge\\
\ge\frac{\prod_{j=1}^c d_j}{\prod_{j=\dim X+1}^n a_j}=\prod_{i=0}^{\dim X} a_i\cdot \frac{\prod_{j=1}^c d_j}{\prod_{s=0}^n a_s}\ge \prod_{i=0}^{\dim X} a_i
\end{multline*}
by Lemma~\ref{lemma:d-2a}(iii).

Suppose that for some $0\le k\le \dim X$ the inequality
$$
a_k\ge 2^{\frac{\dim X+1}{\dim X-k+1}}
$$
holds. Then
$$
\prod_{i=0}^{\dim X}a_i\ge \prod_{i=k}^{\dim X}a_i\ge a_k^{\dim X-k+1}\ge 2^{\dim X+1},
$$
and thus
$$
\left(1+\frac{\dim X+1}{c}\right)^{\frac{c}{\dim X+1}}\ge 2.
$$
The latter means that $\frac{c}{\dim X+1}\ge 1$. On the other hand,
we have $c\le \dim X$ by~\cite[Theorem~1.3]{ChenChenChen}, which gives a contradiction.

Now suppose that $a_1>1$.
To avoid a contradiction with~\cite[Theorem~1.3]{ChenChenChen}
one must have $a_0=1$  and either~\mbox{$a_1=\ldots=a_{\dim X}=2$},
or~\mbox{$a_1=\ldots=a_{\dim X-1}=2$} and~\mbox{$a_{\dim X}=3$}, because otherwise we obtain
$\prod_{i=0}^{\dim X} a_i\ge 2^{\dim X+1}$ and argue as above.
In both cases
one has $c\ge \dim X-1$ by Lemma~\ref{lemma:smooth-condition}.

Assume that $c=\dim X$, so that $c\ge 2$.
Then
\begin{equation}\label{eq:1}
\left(\frac{\sum_{j=1}^c \frac{d_j}{a_{\dim X+j}}}{c}\right)^c\ge \frac{\prod_{j=1}^c d_j}{\prod_{s=\dim X+1}^n a_s}\ge
\prod_{i=0}^{\dim X} a_i \ge 2^{\dim X}=2^c,
\end{equation}
so that
$$
\sum_{i=1}^c \frac{d_i}{a_{\dim X+i}}\ge 2c.
$$
Thus, as in \eqref{eq:Chens}, we get
\begin{multline*}
2c+1\ge 2\dim X+1\ge 2\dim X+2-I(X)\ge a_0+\ldots+a_{\dim X}-I(X)=\\
=\left(\sum_{j=1}^c\frac{d_j}{a_{\dim X+j}}-c\right) a_{\dim X+1}
\ge (2c-c)a_{\dim X+1}=c a_{\dim X+1}.
\end{multline*}
This implies
$$
a_{\dim X+1}\le \frac{2c+1}{c}<3.
$$
This implies $a_{\dim X+1}=2$, which gives a contradiction with
Lemma~\ref{lemma:smooth-condition}.

Finally, assume that $c=\dim X-1$.
As in \eqref{eq:1}, one has
\begin{equation*}
\left(\frac{\sum_{j=1}^c \frac{d_j}{a_{\dim X+j}}}{c}\right)^c \ge 2^{\dim X}=2^{c+1},
\end{equation*}
so that
$$
\sum_{i=1}^c \frac{d_i}{a_{\dim X+i}}\ge 2\sqrt[c]{2}c.
$$
Thus
\begin{multline*}
2c+3\ge 2\dim X+1\ge 2\dim X+2-I(X)\ge a_0+\ldots+a_{\dim X}-I(X)=\\
=\left(\sum_{j=1}^c\frac{d_j}{a_{\dim X+j}}-c\right) a_{\dim X+1}
\ge (2\sqrt[c]{2}c-c)a_{\dim X+1}=c(2\sqrt[c]{2}-1) a_{\dim X+1}.
\end{multline*}
This implies
\begin{equation}\label{eq:sqrt}
a_{\dim X+1}\le \frac{2c+3}{c(2\sqrt[c]{2}-1)}.
\end{equation}
If $c\ge 3$, then \eqref{eq:sqrt} gives $a_{\dim X+1}<3$, so that $a_{\dim X+1}=2$. This is impossible
by Lemma~\ref{lemma:smooth-condition}.
If $c=2$, then \eqref{eq:sqrt} gives
$$
a_{\dim X+1}\le \frac{7}{2(2\sqrt{2}-1)}<2,
$$
which is a contradiction.
If $c=1$, then \eqref{eq:sqrt} gives
$$
a_{\dim X+1}\le \frac{5}{3}<2,
$$
which is again a contradiction.
\end{proof}

\section{Hodge numbers}
\label{section:Hodge}

The idea of description of Hodge numbers
for complete intersections in weighted projective spaces
as dimensions of graded components of particular (bi)graded rings
goes back to~\cite{Gr69},~\cite{St77},~\cite{Do82},~\cite{PS83};
another approach, due to Hirzebruch, can be found in~\cite[Exp. XI, Theorem 2.2]{SGA7}. For complete intersections in toric varieties one can look at~\cite{BC94}.
The way of the computation called the \emph{Cayley trick} is to relate the Hodge structure of a complete intersection to the Hodge structure
of some higher-dimensional hypersurface. We describe this approach following~\cite{Ma99}.

Let $Y$ be a simplicial toric variety of dimension $n$. Let $D_1,\ldots, D_b$ be its prime boundary divisors.
Denote the group of $r$-cycles on $Y$ modulo rational equivalence by $A_r(Y)$.
Consider an $A_{n-1}(Y)$-graded ring $R_0=\CC[x_1,\ldots, x_b]$ with grading defined by
$$
\deg_A\left(\prod_{i=1}^b x_i^{r_i}\right)=\sum_{i=1}^b r_i D_i.
$$
One has $\Spec(R_0)\cong\Aff^b$, and there is a natural correspondence between rays $e_i$ of a fan of $Y$ and variables $x_i$.
Define a subvariety $Z$ in $\mathrm{Spec}(R_0)$ as a union of hypersurfaces~\mbox{$\{\prod x_i=0\mid \ e_i\notin \sigma\}$} over all cones~$\sigma$ of a fan of~$Y$. Then $Y$ is a geometric quotient
of
$$
U=\Spec(R_0)\setminus Z\subset\Aff^b
$$
by the torus
$$
\DD=\mathrm{Hom}_\Z(A_{n-1}(Y),\CC^*).
$$
We call a polynomial~\mbox{$f\in R_0$} homogeneous if all its monomials are of degree $d$ for some~\mbox{$d\in A_{n-1}(Y)$}.
For any homogeneous polynomials $f_1,\ldots, f_c$ their common zero set intersected with $U$ is stable under the action of $\DD$
so they determine a closed subset $X$ in~$Y$.

Consider a ring $R=\CC[x_1,\ldots,x_b,y_1,\ldots,y_c]$. Choose $c$ homogeneous
poly\-no\-mi\-als~\mbox{$f_1,\ldots,f_c\in R_0\subset R$}
with $\deg_A(f_i)=d_i\in A_{n-1}(Y)$. Define a bigrading on $R$ with values in $A_{n-1}(Y)\times \ZZ$ by $\mathrm{bideg}(x_i)=(D_i,0)$ and $\mathrm{bideg}(y_i)=(-d_i,1)$.
Consider the poly\-no\-mi\-al~\mbox{$F=y_1f_1+\ldots+y_cf_c$}. Obviously, one has $\mathrm{bideg}(F)=(0,1)$.
Define a Jacobian ideal
$$
J=\left(\frac{\partial F}{\partial x_1},\ldots,\frac{\partial F}{\partial x_b}, \frac{\partial F}{\partial y_1},\ldots, \frac{\partial F}{\partial y_c}\right)
$$
and a bigraded ring $R(F)=R/J$.

We will assume that the subvariety $X$ defined by the polynomials $f_1,\ldots,f_c$ is \emph{quasi-smooth}. Recall from~\cite[Definition 1.1]{Ma99} that this means that a common zero set
of~\mbox{$f_1,\ldots,f_c$} inside~$U$ is a smooth subvariety of codimension $c$, cf. Definition~\ref{definition:quasi-smooth}.
In this case~$X$ has a pure Hodge structure on its cohomology, see~\cite[\S3]{Ma99}. In particular, one can talk about Hodge numbers $h^{p,q}(X)$.

Define $c_k$ as a number of cones of dimension $k$ in the fan of $Y$.
Put
$$
l_k=\sum_{i=k}^n (-1)^{i-k} \binom{i}{k}c_{n-i}.
$$

\begin{theorem}[{see~\cite[Theorem 10.8 and Remark 10.9]{Da78}}]
\label{theorem:toric cohomology}
Let $Y$ be a
simplicial toric variety of dimension $n$. Then $\dim H^{2k}(Y,\CC)=l_k$
for all $0\le k\le n$.
\end{theorem}

There is an analog of Lefschetz hyperplane section theorem for complete intersections
in simplicial toric varieties, see e.g.~\cite[Proposition~1.4]{Ma99}.
In particular, the only Hodge numbers of such complete intersection $X$ that are not inherited from
the ambient toric variety are $h^{p,q}(X)$ with $p+q=\dim(X)$.

\begin{theorem}[{see \cite[Theorem 3.6]{Ma99}}]
\label{theorem: toric complete intersections}
Let $Y$ be a simplicial toric variety of dimension $n$ and let $D_1,\ldots,D_b$ be its boundary divisors. Let $X\subset Y$ be a quasi-smooth complete intersection of ample divisors
defined by homogeneous polynomials $f_1,\ldots,f_c$ with~\mbox{$\deg_A (f_i)=d_i\in A_{n-1}$}. Suppose that $\dim X=n-c\ge 3$.
Denote
$$
\sum_{r=1}^b D_r-\sum_{s=1}^{c}d_s\in A_{n-1}(X)
$$
by $i(X)$.
Then for $p\neq \frac{n-c+1}{2}$ and $p\neq \frac{n-c}{2}$ one has
$$
h^{n-c-p,p}(X)=\dim R(F)_{\left(-i(X),p\right)}.
$$
For $p= \frac{n-c+1}{2}$ one has
$$
h^{p-1,p}(X)=\dim R(F)_{\left(-i(X),p\right)}+l_{p-c}-l_{p}.
$$
For $p= \frac{n-c}{2}$ one has
$$
h^{p,p}(X)=\dim R(F)_{\left(-i(X),p\right)}+l_p.
$$
\end{theorem}

In a particular case of complete intersections in a weighted projective space
the even cohomology spaces $H^{2k}(Y,\CC)$ are one-dimensional, see~\cite[Corollary 2.3.6]{Do82}.
This allows to simplify Theorem~\ref{theorem: toric complete intersections} in this case.
Recall that for a weighted complete intersection $X$ in $\PP(a_0,\ldots,a_n)$ of multidegree $(d_1,\ldots,d_c)$ we denote the number~\mbox{$\sum a_i-\sum d_j$} by~$I(X)$.

\begin{corollary}
\label{corollary: weighted complete intersection}
Let $\P=\P(a_0,\ldots, a_n)$ be a well formed weighted projective space,
and~\mbox{$X\subset \P$} be a quasi-smooth weighted complete intersection
defined by homogeneous polynomials $f_1,\ldots,f_c$ with $\deg (f_i)=d_i$. Suppose that
$\dim X=n-c\ge 3$.
Then for $p\neq \frac{n-c}{2}$ one has
$$
h^{n-c-p,p}(X)=\dim R(F)_{\left(-I(X),p\right)}
$$
and for $p=\frac{n-c}{2}$ one has
$$
h^{p,p}(X)=\dim R(F)_{\left(-I(X),p\right)}+1.
$$
\end{corollary}

\begin{example}
Consider the weighted projective space $\P=\P(1,1,1,1,1,1,3)$ with weighted
homogeneous coordinates $x_0,\ldots,x_6$, where the weights of $x_0,\ldots,x_5$ equal $1$,
and the weight of $x_6$ equals $3$. Let $X$ be a (general) weighted complete intersection of hypersurfaces of degrees $2$ and $6$ in~$\P$
given by polynomials $f_1$ and $f_2$, respectively.
Thus~\mbox{$F=y_1f_1+y_2f_2$} and
$$
J=\left(\frac{\partial F}{\partial x_0},\ldots,\frac{\partial F}{\partial x_6}, f_1,f_2\right).
$$
One has
\begin{multline*}
\bideg (x_0)=\ldots =\bideg(x_5)=(1,0), \bideg(x_6)=(3,0),\\
\bideg(y_1)=(-2,1), \bideg(y_2)=(-6,1).
\end{multline*}

Since $I(X)=1$, one gets $h^{1,3}(X)=\dim R(F)_{(-1,1)}$. The component $R(F)_{(-1,1)}$ is generated by polynomials
of type $g_1y_1+g_5y_2$, where $g_s$ are polynomials of degrees $s$ in $x_i$.
There are $6$ parameters for $g_1$ and $\binom{10}{5}+\binom{7}{5}=273$ parameters for $g_5$,
so $R_{(-1,1)}=279$.
There are no polynomials from $R_{(-1,1)}$ that are divisible by $f_2$, and $\binom{8}{5}=56$ parameters for polynomials
in $x_0,\ldots, x_5, y_2$ that
are divisible by $f_1$. Up to scaling there is a unique polynomial that is divisible by $f_1$ and $x_6$. Moreover, one has $\bideg(\frac{\partial F}{\partial x_i})=(-1,1)$
for $i=1,\ldots,6$, so there are $6$ parameters for polynomials from $R_{(-1,1)}$ that are divisible by $\frac{\partial F}{\partial x_i}$.
Similarly, one has $\bideg(\frac{\partial F}{\partial x_6})=(-3,1)$, so there are $\binom{7}{5}=21$ parameters for polynomials from $R_{(-1,1)}$ that are divisible by
$\frac{\partial F}{\partial x_6}$. One of them, namely $\frac{\partial F}{\partial x_6}f_1$, is already taken into account.
Thus
$$
h^{1,3}(X)=\dim R(F)_{(-1,1)}= \dim R_{(-1,1)}-56-1-6-21+1=279-56-1-6-21+1=196.
$$
\end{example}

\begin{remark}
The method given by
Corollary~\ref{corollary: weighted complete intersection} enables one to
study weighted complete intersections with interesting behavior
of Hodge numbers. Smooth Fano weighted complete
intersections with few non-vanishing Hodge numbers
(more precisely, with small \emph{Hodge complexity}) were
classified in~\cite{PSh18}.
It turns out that they always
have certain interesting properties which can be described
in terms of semi-orthogonal decompositions
of derived categories of coherent sheaves.
\end{remark}

One can obtain the following elegant formula for $\dim H^0(X,\O_X(k))$.

\begin{theorem}[{\cite[Theorem 3.4.4]{Do82}}]
\label{theorem:elegant}
Let $X$ be a quasi-smooth well formed weighted complete intersection of multidegree $(d_1,\ldots,d_c)$ in $\PP(a_0,\ldots,a_n)$.
Then
$$
\sum_{k=0}^\infty\left(\dim H^0(X,\O_X(k))\right)t^k=\frac{\prod_{s=0}^c\left(1-t^{d_s}\right)}{\prod_{r=0}^n\left(1-t^{a_r}\right)}.
$$
\end{theorem}

For a weighted projective space $\P(a_0,\ldots, a_n)$ with weighted homogeneous coordinates~\mbox{$x_0,\ldots,x_n$} denote by $P(r)$ the dimension of the vector space of (weighted) homogeneous polynomials in $x_0,\ldots,x_n$ of (weighted) degree $r$.
Theorem~\ref{theorem:elegant} implies the following.

\begin{corollary}\label{corollary:ashnol'}
Let $X$ be a quasi-smooth well formed weighted complete intersection of multidegree $(d_1,\ldots,d_c)$ in $\PP(a_0,\ldots,a_n)$.
Then
$$
\dim H^0(X,-K_X)=\sum_{s=0}^c (-1)^s \sum_{1\le k_1<\ldots<k_s\le c} P(I(X)-d_{k_1}-\ldots-d_{k_s}).
$$
\end{corollary}
\begin{proof}
By Lemma~\ref{lemma:adjunction}, we are going to compute
$$
\dim H^0(X,-K_X)=\dim H^0\big(X,\O_X(I(X))\big).
$$
This number equals the coefficient at $t^{I(X)}$ in the right hand side of the equality in Theorem~\ref{theorem:elegant}.
Note that
$$
\prod_{s=0}^c\left(1-t^{d_s}\right)=\sum_{s=0}^c (-1)^s \sum_{1\le k_1<\ldots<k_s\le c} t^{d_{k_1}+\ldots+d_{k_s}}.
$$
On the other hand, the coefficient at $t^{I(X)-d_{k_1}-\ldots-d_{k_s}}$ in $\prod_{r=0}^n\left(1-t^{a_r}\right)$
is equal to the number~\mbox{$P(I(X)-d_{k_1}-\ldots-d_{k_s})$}, which easily implies the required assertion.
\end{proof}

Another approach to compute Hodge numbers for complete intersections in (usual) projective spaces is due to F.\,Hirzebruch.
Define
$$
H(d)=\frac{(y+1)^{d-1}-(z+1)^{d-1}}{(z+1)^dy-(y+1)^dz}=\frac{d-1+\binom{d-1}{2}(y+z)+\binom{d-1}{3}(y^2+yz+z^2)+\ldots}
{1-\binom{d}{2}yz-\binom{d}{3}yz(y+z)+\ldots}
$$
and
$$
H(d_1,\ldots,d_c)=\sum_{Q\subset [1,c],Q\neq \varnothing} \left((y+1)(z+1)\right)^{|Q|-1}\prod_{i\in Q}H(d_i),
$$
where $|Q|$ is a number of elements of $Q$.
If $F$ is a formal series in two variables $y$ and $z$, we denote by $F^{(m)}$ the sum of monomials in $F$ of homogeneous degree $m$.

\begin{theorem}[{see~\cite[Exp. XI, Th\'eor\`eme 2.3]{SGA7}}]
\label{theorem:Hirzebruch}
Let $X$ be a smooth complete intersection of hypersurfaces of degrees $d_1,\ldots,d_c$
in~$\P^n$.
Put $m=\dim X=n-c$. Then
$$
\sum h^{p,m-p}(X)y^pz^{m-p}=\left(H(d_1,\ldots,d_c)+\delta y^{\frac{m}{2}}z^{\frac{m}{2}}\right)^{(m)},
$$
where $\delta=1$ if $m$ is even and $\delta=0$ if $m$ is odd.
\end{theorem}

\begin{remark}
There is a conjectural approach to description of Hodge numbers of Fano varieties via their Landau--Ginzburg models, see~\cite{KKP14}. It was verified for del Pezzo surfaces
(see~\cite{LP16}) and Fano threefolds (see~\cite{Prz13} and~\cite{ChP18}); for smooth toric varieties see~\cite{Ha16}.
Its reformulation in terms of toric Landau--Ginzburg models for one of the Hodge numbers was checked 
for complete intersections in projective spaces (see~\cite{PS15}).
\end{remark}

\section{Dimensions $4$ and $5$}
\label{section:tables}

In this section we provide a classification of smooth well formed Fano weighted complete intersections of dimensions $4$ and $5$.
To simplify
conventions, we exclude the projective space (which is a codimension~$0$
smooth Fano complete intersection in itself) from our lists.

We find the weighted complete intersections we are interested in
by a straightforward check of all possible weights and degrees. Namely, we know
that if there is a smooth well formed Fano weighted complete intersection of dimension~$r$ and multidegree~\mbox{$(d_1,\ldots,d_{n-r-1})$} in~\mbox{$\PP(a_0,\ldots,a_n)$} that is not an intersection with a linear cone, then~\mbox{$n\le 2r$} by~\cite[Theorem~1.3]{ChenChenChen}.
Next, by Theorem~\ref{theorem:main} we have
$a_i\le n$ for all~\mbox{$0\le i\le n$},
so that for any given dimension $r$ we have to check only a finite number of possible weights
$a_0,\ldots,a_n$; one can further reduce the number of cases to be checked
using Lemma~\ref{lemma:middle-bound}.
Finally, the degrees $d_j$ can be bounded from above in terms of
$a_i$ (see Theorem~\ref{theorem:main} or Lemma~\ref{lemma:adjunction}), so we have
only a finite number of possible collections~\mbox{$d_1,\ldots,d_{n-r-1}$} to check;
here the number of cases can also be reduced using
the lower bounds from Lemma~\ref{lemma:d-2a}.
Note that at this step we obtain only necessary conditions on the weights and degrees,
and in each case one has to check that there actually exists a weighted complete
intersection with the corresponding parameters,
and that it is smooth. This can be done by writing down a general
equation of a weighted complete intersection in the few remaining families.
The Hodge numbers of our weighted complete intersections
are computed using Corollary~\ref{corollary: weighted complete intersection},
and sometimes Theorem~\ref{theorem:Hirzebruch} when the latter is more convenient to apply.
Since the Hodge numbers are constant in smooth families,
for these computations it is enough to pick one example in each of the families.
In principle, all this can be done in an arbitrary given dimension, although the number of arising
Fano varieties becomes rather large in higher dimensions. For instance, there are~$22$
deformation families of smooth well formed four-dimensional Fano weighted complete intersections
that are not intersections with a linear cone (excluding the variety~$\PP^4$, as usual);
there are~$35$ such deformation families of fivefolds; and there are~$72$ such families of sixfolds
(here the number of families of codimension $1$, $2$, $3$, $4$, $5$, and~$6$
weighted complete intersections equals~$18$, $31$, $15$, $5$, $2$,
and~$1$, respectively).

Let $X$ be a smooth well formed Fano weighted complete intersection of
multi\-deg\-ree~\mbox{$(d_1,\ldots,d_c)$} in $\PP=\PP(a_0,\ldots,a_n)$ of dimension~\mbox{$n-c\ge 3$}.
Important invariants of~$X$ are its anticanonical degree $(-K_X)^{\dim X}$, the dimension $h^0(-K_X)$, and the index~$I(X)$, which is
defined as the maximal number $i$ such that $K_X$ is divisible by~$i$ in~\mbox{$\Pic (X)$}.
Since~\mbox{$\dim X\ge 3$}, the class of the line bundle $\mathcal{O}_{\P}(1)\vert_X$ is not divisible in~$\Pic(X)$, see~\cite[Remark~4.2]{Okada2}. Therefore, by Lemma~\ref{lemma:adjunction}
one has
$$
I(X)=\sum a_i-\sum d_j.
$$
For the anticanonical degree of $X$ one has
$$
(-K_X)^{\dim(X)}=\frac{\prod d_j}{\prod a_i}\cdot I(X)^{\dim(X)}.
$$
The number $h^0(-K_X)$ can be computed by Corollary~\ref{corollary:ashnol'}.

We will use the abbreviation
\begin{equation*}
(a_0^{k_0},\ldots,a_m^{k_m})=
(\underbrace{a_0,\ldots,a_0}_{k_0\ \text{times}},\ldots,\underbrace{a_m,\ldots,a_m}_{k_m\ \text{times}}),
\end{equation*}
where $k_0,\ldots,k_m$ will be allowed to be
any positive integers. If some $k_i$ is equal to $1$ we drop it for simplicity.

Table~\ref{table:Fano-dim-4} contains a list of all smooth well formed Fano weighted complete intersections of dimension $4$ that are not intersections with linear cones. This list was obtained in~\cite[Proposition~2.2.1]{Kuchle-Geography}, cf.~\cite[\S1.3]{BrownKasprzyk}.

\begin{center}
\begin{longtable}{||c|c|c|c|c|c|c|c||}
  \hline
  No. & $I$ & $\P$ & Degrees & $K^4$ & $h^0(-K)$ & $h^{1,3}$ & $h^{2,2}$
  \\
  \hline
  \hline
  \endhead
1 & 1 & $\PP(1^3,2^2,3^2)$ & 6,6 & 1 & 3 & 107 & 503
\\\hline
2 & 1 & $\PP(1^4,2,5)$ & 10 & 1 & 4 & 412 & 1801
\\\hline
3 & 1 & $\PP(1^4,2^2,3)$ & 4,6 & 2 & 4 & 121 & 572
\\\hline
4 & 1 & $\PP(1^5,4)$ & 8 & 2 & 5 & 325 & 1452
\\\hline
5 & 1 & $\PP(1^5,2)$ & 6 & 3 & 5 & 156 & 731
\\\hline
6 & 1 & $\PP(1^5,2^2)$ & 4,4 & 4 & 5 & 75 & 378
\\\hline
7 & 1 & $\PP(1^6,3)$ & 2,6 & 4 & 6 & 196 & 912
\\\hline
8 & 1 & $\PP^5$ & 5 & 5 & 6 & 120 & 581
\\\hline
9 & 1 & $\PP(1^6,2)$ & 3,4 & 6 & 6 & 71 & 364
\\\hline
10 & 1 & $\PP^6$ & 2,4 & 8 & 7 & 77 & 394
\\\hline
11 & 1 & $\PP^6$ & 3,3 & 9 & 7 & 49 & 267
\\\hline
12 & 1 & $\PP^7$ & 2,2,3 & 12 & 8 & 42 & 236
\\\hline
13 & 1 & $\PP^8$ & 2,2,2,2 & 16 & 9 & 27 & 166
\\\hline\hline
14 & 2 & $\PP(1^5,3)$ & 6 & 32 & 15 & 70 & 382
\\\hline
15 & 2 & $\PP^5$ & 4 & 64 & 21 & 21 & 142
\\\hline
16 & 2 & $\PP^6$ & 2,3 & 96 & 27 & 8 & 70
\\\hline
17 & 2 & $\PP^7$ & 2,2,2 & 128 & 33 & 3 & 38
\\\hline\hline
18 & 3 & $\PP(1^4,2,3)$ & 6 & 81 & 25 & 24 & 161
\\\hline
19 & 3 & $\PP(1^5,2)$ & 4 & 162 & 40 & 5 & 52
\\\hline
20 & 3 & $\PP^5$ & 3 & 243 & 55 & 1 & 21
\\\hline
21 & 3 & $\PP^6$ & 2,2 & 324 & 70 & 0 & 8
\\\hline\hline
22 & 4 & $\PP^5$ & 2 & 512 & 105 & 0 & 2
\\
  \hline
\caption[]{\label{table:Fano-dim-4} Fourfold Fano weighted complete intersections.}
\end{longtable}
\end{center}

\begin{remark}\label{remark:Kuchle}
Note that there is a misprint in the first line of the table on~\cite[p.~50]{Kuchle-Geography}:
the varieties described there should be understood as complete intersections of
\emph{two} hypersurfaces of weighted degree $6$ in $\P(1^3,2^2,3^2)$. This corresponds to
family No.~1 in Table~\ref{table:Fano-dim-4}.
\end{remark}

\begin{remark}\label{remark:linear-cone}
The numerical data listed in Table~\ref{table:Fano-dim-4}
does not describe \emph{every} variety
in the corresponding deformation family.
For example, a quartic in $\P^5$
can be seen as a complete intersection of bidegree ($2,4$)
in $\P(1^6,2)$, that is an intersection with a linear cone.
A non-general variety of the latter
type can be contained in a hypersurface of weighted degree~$2$
whose equation does not depend on the variable of weight~$2$; such complete
intersections cannot be embedded as quartics in $\P^5$.
\end{remark}

Looking at the anticanonical degrees and dimensions of anticanonical linear systems
of varieties from Table~\ref{table:Fano-dim-4}, we see that varieties from different
families are never isomorphic to each other. Similarly, none of them
is isomorphic to any of the smooth Fano fourfolds
that are zero loci of sections of homogeneous vector bundles on Grassmannians,
see~\cite[Theorem~4.2.1]{Kuchle-Geography}.

The information
concerning rationality of varieties listed in Table~\ref{table:Fano-dim-4}
that we are aware of is summarized in Table~\ref{table:Fano-dim-4-rationality}.
In the first column, we list the number of the family according to
Table~\ref{table:Fano-dim-4}. In the second column, we give the condition for
a variety in the family to be not stably rational (this might be either ``very general'' or ``none'').
In the third column, we give the condition for a variety in the family to be non-rational
(in most cases this information is obtained just from the previous column).
In the fourth column, we give the condition for a variety in the family to be rational.
In all these columns empty cells mean that we know nothing about
the corresponding property for the varieties in the family.
Finally, in the last column we give references for the corresponding theorems.

\def\arraystretch{1.4}
\begin{center}
\begin{longtable}{||c|c|c|c|c||}
  \hline
  No. & Not stably rational & Non-rational & Rational & Reference
  \\
  \hline
  \hline
\endhead
 $2$ & very general & very general & & \cite[Corollary~1.4]{Okada1}\\
 \hline
 $4$ & very general & very general & & \cite[Theorem~1.1]{Okada1}\\
 \hline
 $5$ & very general & very general & & \cite[Corollary~1.4]{Okada1}\\
 \hline
 $8$ & very general & general &  & \cite{Totaro}, \cite{Schreieder}, \cite{Pukhlikov1}\\
 \hline
 $14$ & very general  & very general & & \cite[Theorem~1.1]{Okada1}\\
 \hline
 $15$ & very general & very general & & \cite{Totaro}, \cite{Schreieder}\\
 \hline
 $17$ & very general &  very general & some & 
 \cite{HPT17}\\
 \hline
 $18$ & very general & very general & & \cite[Theorem~1.3]{Okada1}\\
 \hline
 $19$ & very general & very general & & \cite{HPT}\\
 \hline
 $20$ &  &  & some & 
 \cite{RussoStagliano}\\
 \hline
 $21$ & none & none & all & 
 projection from a line \\
 \hline
 $22$ & none & none & all & 
 projection from a point \\
 \hline
 \caption[]{\label{table:Fano-dim-4-rationality} Rationality for fourfold Fano weighted complete intersections.}
\end{longtable}
\end{center}

Now we consider five-dimensional weighted complete intersections.
Table~\ref{table:Fano-dim-5} contains a list of all smooth well formed Fano weighted complete intersections of dimension $5$ that are not intersections with linear cones.

\def\arraystretch{1.4}
\begin{center}
\begin{longtable}{||c|c|c|c|c|c|c|c||}
  \hline
  No. & $I$ & $\P$ & Degrees & $-K^5$ &  $h^0(-K)$ & $h^{1,4}$ & $h^{2,3}$
  \\
  \hline
  \hline
\endhead
1   & $1$ & $\PP(1^5,2,3,3)$ & $6,6$ & $2$ & 5 & 354 & 4594
 \\
  \hline
2 & $1$ & $\PP(1^6,5)$ & $10$ & $2$ & 6 & 1996 & 24576
 \\
  \hline
3   & $1$ & $\PP(1^6,2,3)$ & $4,6$ & $4$ & 6 & 359 & 4758
 \\
  \hline
4   & $1$ & $\PP(1^7,4)$ & $2,8$ & $4$ & 7 & 1386 & 15771
 \\
  \hline
5 & $1$ & $\PP^6$ & $6$ & $6$ & 7 & 455 & 6055
 \\
  \hline
6   & $1$ & $\PP(1^7,2)$ & $4,4$ & $8$ & 7 & 168 & 2383
 \\
   \hline
7  & $1$ & $\PP(1^8,3)$ & 2,2,6 & $8$ & 8 & 568 & 7571
\\
  \hline
8 & $1$ & $\PP^7$ & $2,5$ & $10$ & 8 & 294 & 4074
 \\
  \hline
9 & $1$ & $\PP^7$ & $3,4$ & $12$ & 8 & 147 & 2142
 \\
   \hline
10  & $1$ & $\PP^8$ & 2,2,4 & $16$ & 9 & 156 & 2295
\\
   \hline
11  & $1$ & $\PP^8$ & 2,3,3 & $18$ & 9 & 88 & 1364
\\
   \hline
12  & $1$ & $\PP^9$ & 2,2,2,3 & $24$ & 10 & 72 & 1155
\\
\hline
13  & $1$ & $\PP^{10}$ & 2,2,2,2,2 & $32$ & 11 & 44 & 759
\\
  \hline\hline
14   & $2$ & $\PP(1^4,2,2,3,3)$ & $6,6$ & $32$ & 12 & 122 & 1920
 \\
  \hline
15   & $2$ & $\PP(1^5,2,5)$ & $10$ & $32$ & 16 & 790 & 11020
 \\
  \hline
16   & $2$ & $\PP(1^5,2,2,3)$ & $4,6$ & $64$ & 17 & 117 & 1936
 \\
  \hline
17   & $2$ & $\PP(1^6,4)$ & $8$ & $64$ & 21 & 462 & 6891
 \\
  \hline
18   & $2$ & $\PP(1^6,2)$ & $6$ & $96$ & 22 & 147 & 2457
 \\
  \hline
19   & $2$ & $\PP(1^6,2,2)$ & $4,4$ & $128$ & 23 & 44 & 867
 \\
  \hline
20   & $2$ & $\PP(1^7,3)$ & $2,6$ & $128$ & 27 & 183 & 3072
 \\
  \hline
21 & $2$ & $\PP^6$ & $5$ & $160$ & 28 & 84 & 1554
 \\
  \hline
22   & $2$ & $\PP(1^7,2)$ & $3,4$ & $192$ & 29 & 35 & 742
 \\
  \hline
23  & $2$ & $\PP^7$ & $2,4$ & $256$ & 35 & 36 & 783
  \\
\hline
24 & $2$ & $\PP^7$ & 3,3 & $288$ & 36 & 16 & 410
 \\
    \hline
25  & $2$ & $\PP^8$ & 2,2,3 & $384$ & 43 & 11 & 316
\\
    \hline
26  & $2$ & $\PP^9$ & 2,2,2,2 & $512$ & 51 & 4 & 159
\\
  \hline\hline
27   & $3$ & $\PP(1^6,3)$ & $6$ & $486$ & 57 & 56 & 1246
 \\
  \hline
28 & $3$ & $\PP^6$ & $4$ & $972$ & 84 & 7 & 266
 \\
  \hline
29  & $3$ & $\PP^7$ & 2,3 & $1458$ & 111 & 1 & 83
\\
  \hline
30  & $3$ & $\PP^8$ & 2,2,2 & $1944$ & 138 & 0 & 27
\\
  \hline\hline
31   & $4$ & $\PP(1^5,2,3)$ & $6$ & $1024$ & 91 & 16 & 505
 \\
  \hline
32   & $4$ & $\PP(1^6,2)$ & $4$ & $2048$ & 147 & 1 & 90
 \\
  \hline
33 & $4$ & $\PP^6$ & $3$ & $3072$ & 203 & 0 & 21
 \\
  \hline
34  & $4$ & $\PP^7$ & $2,2$ & $4096$ & 259 & 0 & 3
 \\
  \hline\hline
35 & $5$ & $\PP^6$ & $2$ & $6250$ & 378 & 0 & 0
 \\
  \hline
\caption[]{\label{table:Fano-dim-5} Fivefold Fano weighted complete intersections.}
\end{longtable}
\end{center}

Similarly to Remark~\ref{remark:linear-cone}, the numerical data listed in Table~\ref{table:Fano-dim-5}
does not describe every variety
in the corresponding deformation family, but only a general one.

Looking at the anticanonical degrees and dimensions of anticanonical linear systems
of varieties from Table~\ref{table:Fano-dim-5}, we see that varieties from different
families are never isomorphic to each other.

The information
concerning rationality of varieties listed in Table~\ref{table:Fano-dim-5}
that we are aware of is summarized in Table~\ref{table:Fano-dim-5-rationality},
with the notation similar to that of Table~\ref{table:Fano-dim-4-rationality}.

\def\arraystretch{1.4}
\begin{center}
\begin{longtable}{||c|c|c|c|c||}
  \hline
  No. & Not stably rational & Non-rational & Rational & Reference
  \\
  \hline
  \hline
\endhead
 $2$ & very general & very general & & \cite[Theorem~1.1]{Okada1}\\
 \hline
 $5$ & very general & general & & \cite{Totaro}, \cite{Schreieder}, \cite{Pukhlikov1}\\
 \hline
 $9$ &  & some &  & \cite{Pukhlikov2}\\
 \hline
 $15$ & very general & very general & & \cite[Theorem~1.3]{Okada1}\\
 \hline
 $17$ & very general & very general &  & \cite[Theorem~1.1]{Okada1}\\
 \hline
 $18$ & very general & very general & & \cite[Theorem~1.3]{Okada1}\\
 \hline
 $21$ & very general & very general &  & \cite{Schreieder}\\
 \hline
 $27$ & very general & very general &  & \cite[Theorem~1.1]{Okada1}\\
 \hline
 $30$ & none & none & all & \cite[Corollary~5.1]{Tyurin} \\
 \hline
 $34$ & none & none & all & 
 projection from a line \\
 \hline
 $35$ & none & none & all & 
 projection from a point \\
 \hline
 \caption[]{\label{table:Fano-dim-5-rationality} Rationality for fivefold Fano weighted complete intersections.}
\end{longtable}
\end{center}

It would be interesting to study birational geometry of weighted complete intersections
from Tables~\ref{table:Fano-dim-4} and \ref{table:Fano-dim-5} that are not covered by
Tables~\ref{table:Fano-dim-4-rationality} and~\ref{table:Fano-dim-5-rationality}.
Also, it would be interesting to study automorphism groups of Fano varieties from Tables~\ref{table:Fano-dim-4} and \ref{table:Fano-dim-5}, cf.~\mbox{\cite[\S\,A.2]{ProkhorovShramov}}. In particular, it would be interesting to find weighted Fano complete intersections
acted on by relatively large automorphism groups, cf.~\cite{PS16}.

Using the list of index $1$ Fano fivefolds provided in Table~\ref{table:Fano-dim-5}, one can compile the list of
smooth well formed Calabi--Yau weighted complete intersections of dimension $4$
that are not intersections with linear cones (cf.~\cite{Og91}). Namely, if there is a smooth Calabi--Yau weighted complete intersection of multidegree~\mbox{$d_1,\ldots,d_c$} in $\P(a_0,\ldots,a_n)$, then a general complete intersection
of multidegree~\mbox{$d_1,\ldots,d_c$} in $\P(1,a_0,\ldots,a_n)$ is a smooth Fano variety.
Note that the converse also holds: if there is a smooth Fano normalized weighted complete intersection of multidegree~\mbox{$d_1,\ldots,d_c$} in $\P(a_0,\ldots,a_n)$ that has index~$1$, then it follows from~\cite[Theorem~1.2]{PST}
that $a_0=1$ and a general weighted complete intersection of multidegree
$d_1,\ldots,d_c$ in~\mbox{$\P(a_1,\ldots,a_n)$} is a smooth Calabi--Yau variety.
For other partial classification results concerning Calabi--Yau threefolds see \cite{AESZ05},
\cite{IMOU}, \cite{IIM}, \cite{Benedetti}, and references therein.

\appendix

\section{Optimization}
\label{section:optimization}

The purpose of this section is to prove some bounds on the values of linear
functions on special subsets of~$\R^m$.

Let $L$ be a real-valued function on a set $\Omega$. We say that
$L$ \emph{attains its maximum in $\Omega$} if $L$ is bounded and
there is a point $P$ in $\Omega$ such that $L(P)=\sup_{P'\in\Omega}L(P')$;
in this case we also say that $L$ attains its maximum in $\Omega$ at $P$.

\begin{lemma}\label{lemma:attains-preliminary}
Let $N$ and $c$ be positive integers, and $M$ be a real number.
Consider the closed subset of $\R^{N+c}$ with
coordinates~\mbox{$A_1,\ldots,A_N,D_1,\ldots, D_c$} defined by inequalities
\begin{equation*}
M\ge A_1\ge\ldots\ge A_N\ge 0, \quad D_1\ge\ldots\ge D_c\ge 0,
\end{equation*}
and let $\Omega'$ be a non-empty closed subset therein.
Put
$$
L(A_1,\ldots,A_N,D_1,\ldots,D_c)=\sum_{i=1}^N A_i-\sum_{j=1}^c D_j.
$$
Then the function $L$
attains its maximum in $\Omega'$.
\end{lemma}
\begin{proof}
Let~$\bar{P}$
be some point of $\Omega'$, and put $\bar{L}=L(\bar{P})$.
If $D_1>NM-\bar{L}$,
then
$$
L(A_1,\ldots,D_c)=\sum_{i=1}^N A_i-\sum_{j=1}^c D_j\le
NA_1-D_1< N(A_1-M)+\bar{L}\le \bar{L}.
$$
Thus $L$ attains its maximum in $\Omega'$ if and only if
it attains its maximum in the closed subset
$$
\Omega''=\Omega'\cap\{(A_1,\ldots,D_c)\mid D_1\le NM-\bar{L}\}
$$
containing $\bar{P}$.
It remains to notice that $\Omega''$ is a compact subset of $\R^{N+c}$.
\end{proof}

\begin{lemma}\label{lemma:attains}
Let $N$ and $c$ be positive integers such that $N>c$,
and $\alpha$ be a real number.
Let $\widetilde{\Omega}$ be a subset of $\R^{N+c}$ with
coordinates $A_1,\ldots,A_N,D_1,\ldots, D_c$ defined by inequalities
\begin{gather*}
A_1\cdot\ldots\cdot A_N\le D_1\cdot\ldots\cdot D_c,\\
A_1\ge\ldots\ge A_N\ge 1, \quad D_1\ge\ldots\ge D_c,\\
D_1\ge 2A_1, \ D_2\ge A_2,\ \ldots, \ D_c\ge A_c.
\end{gather*}
Put
$$
L(A_1,\ldots,A_N,D_1,\ldots,D_c)=\sum_{i=1}^N A_i-\sum_{j=1}^c D_j+\alpha.
$$
Let $\Omega\subset\widetilde{\Omega}$ be a non-empty
closed subset. Then the function $L$
attains its maximum in $\Omega$.
\end{lemma}
\begin{proof}
It is enough to prove the assertion for $\alpha=0$.
Let $\bar P$
be some point of $\Omega$, and put~\mbox{$\bar L=L(\bar P)$}.

Suppose that $(A_1,\ldots,D_c)\in\Omega$.
If $A_{c+1}<(A_1+\bar L)\cdot (N-c)^{-1}$, then
\begin{multline*}
L(A_1,\ldots,D_c)=\sum_{i=1}^N A_i-\sum_{j=1}^c D_j=
(A_1-D_1)+\sum_{j=2}^c(A_j-D_j)+\sum_{i=c+1}^N A_i\le
\\ \le -A_1+(N-c)A_{c+1}<-A_1+A_1+\bar L=\bar L.
\end{multline*}
Thus $L$ attains its maximum in $\Omega$ if and only if
it attains its maximum in the closed subset
$$
\Omega^{(1)}=\Omega\cap\left\{(A_1,\ldots, D_c)\mid A_{c+1}\ge\frac{A_1+\bar L}{N-c}
\right\}
$$
containing $\bar P$.

By Lemma~\ref{lemma:attains-preliminary}, the function $L$ attains its maximum in
the closed subset
$$
\Omega'=\Omega^{(1)}\cap\{(A_1,\ldots,D_c)\mid A_1\le -\bar L\}
$$
provided that $\Omega'$ is not empty.
Therefore, to prove that $L$ attains its maximum in $\Omega^{(1)}$ it is enough to show
that either $L$ attains its maximum
in the closed subset
$$
\Omega^{(2)}=\Omega^{(1)}\cap\{(A_1,\ldots,D_c)\mid A_1\ge -\bar L\},
$$
or $L(P)<\bar L$ for every point $P\in\Omega^{(2)}$. The latter case takes place
in particular when~$\Omega^{(2)}$ is empty.

Suppose that $\Omega^{(2)}$ is not empty, and let $(A_1,\ldots,D_c)\in\Omega^{(2)}$. Then
$$
D_1^c\ge D_1\cdot\ldots\cdot D_c\ge A_1\cdot\ldots\cdot A_N\ge
A_1\cdot\ldots\cdot A_{c+1}\ge \left(\frac{A_1+\bar L}{N-c}\right)^{c+1},
$$
so that
$$
D_1\ge \left(\frac{1}{N-c}\right)^{\frac{c+1}{c}}\cdot
(A_1+\bar L)^{\frac{c+1}{c}}.
$$
One has
\begin{multline}\label{eq:M}
L(A_1,\ldots,D_c)=\sum_{i=1}^N A_i-\sum_{j=1}^c D_j\le N\cdot A_1-D_1
=N(A_1+\bar L)-D_1-N\bar L\le\\
\le (A_1+\bar L)\cdot \left(N-
\left(\frac{1}{N-c}\right)^{\frac{c+1}{c}}\cdot
(A_1+\bar L)^{\frac{1}{c}}\right)-N\bar L.
\end{multline}
Put
$$
M=\max\left\{|(N+1)\bar L|-\bar L, (N+1)^c\cdot(N-c)^{c+1}-\bar L\right\}.
$$
If $A_1>M$, then the right hand side of~\eqref{eq:M}
is less than $\bar L$.
Thus to complete the proof it is enough to show that either $L$ attains its maximum in the closed subset
$$
\Omega^{(3)}=\Omega^{(2)}\cap\{(A_1,\ldots, D_c)\mid A_1\le M\},
$$
or $\Omega^{(3)}$ is empty. Now everything follows from Lemma~\ref{lemma:attains-preliminary}.
\end{proof}

Similarly to Lemma~\ref{lemma:attains}, we prove
the following.

\begin{lemma}\label{lemma:bb2dd-attains}
Let $l$ and $r$ be positive integers such that
$l>r$, and let $\alpha$ be a real number. Let $\Omega\subset\R^2_{B,D}$ be defined by inequalities
$B^l\le 2D^r$ and $D\ge B\ge 1$.
Put
$$
L(B,D)=lB-rD+\alpha.
$$
Then the function $L$
attains its maximum in $\Omega$.
\end{lemma}
\begin{proof}
It is enough to prove the assertion for $\alpha=0$.
Let $\bar P$
be some point of $\Omega$, and put~\mbox{$\bar L=L(\bar P)$}.

Put
$$
M=\max\left\{-\frac{2\bar L}{r}, \left(\frac{4l}{r}\right)^{\frac{l}{l-r}}\right\}.
$$
Suppose that $(B,D)\in\Omega$ and $D>M$.
In particular, we have
$$
D^{\frac{r-l}{l}}<\frac{r}{4l}.
$$
We also know that $B\le \sqrt[l]{2}D^{\frac{r}{l}}$. Hence
$$
L(B,D)=lB-rD\le l\sqrt[l]{2}D^{\frac{r}{l}}-rD<\left(\frac{2l}{r}D^{\frac{r-l}{l}}-1\right)\cdot rD<
-\frac{r}{2}D<\bar L.
$$
Thus $L$ attains its maximum in $\Omega$ if and only if
it attains its maximum in the closed subset
$$
\Omega^{(1)}=\Omega\cap\left\{(B,D)\mid D\le M \right\}
$$
containing $\bar P$. Since $B\le D$ in $\Omega$, the set $\Omega^{(1)}$ is compact,
and the required assertion follows.
\end{proof}

The following theorem will be our main technical tool to find the points
where certain functions attain their maximal values. It is well known as
the \emph{method of Lagrange multipliers}, or sometimes the \emph{Kuhn--Tusker Theorem}.

\begin{theorem}[{see~\cite[p. 503]{Va92} or~\cite[Theorem M.K.2]{MWG95}}]
\label{theorem:Lagrange}
Let $G_1,\ldots,G_p$ be differentiable functions on
$\R^m$ with coordinates $x_1,\ldots,x_m$.
Let $\Omega\subset\R^m$ be a subset defined by
inequalities $G_i\le 0$, $1\le i\le p$. Let $L$ be a differentiable
function on $\R^m$. Suppose that $L$ attains its maximum
in $\Omega$ at a point $P$. Then
$$
\left(\frac{\partial L}{\partial x_1},\ldots,
\frac{\partial L}{\partial x_m}\right)[P]=
\sum\limits_{i=1}^p
\lambda_i\left(\frac{\partial G_i}{\partial x_1},\ldots,
\frac{\partial G_i}{\partial x_m}\right)[P]
$$
for some non-negative numbers $\lambda_i$. Moreover, if
for some $j$ one has $\lambda_j\neq 0$,
then~\mbox{$G_j(P)=0$}.
\end{theorem}

To proceed we need to establish some elementary inequalities
that will be used in the proof of Proposition~\ref{proposition:Lagrange}.

\begin{lemma}\label{lemma:bb2dd}
Let $l\ge 1$ be an integer, and $r$ be a non-negative integer.
Let $\Omega\subset\R^2_{B,D}$ be defined by inequalities
$B^l\le 2D^r$ and $D\ge B\ge 1$.
Put
$$
L(B,D)=lB-rD-(l+1).
$$
Then $L$ is non-positive on $\Omega$. Moreover, $L(B,D)$ is negative unless $l=1$, $r=0$, and~\mbox{$B=2$}.
\end{lemma}
\begin{proof}
Suppose that $l\le r$. Then
$$
L(B,D)=lB-rD-(l+1)\le (l-r)B-(l+1)<0
$$
for all $(B,D)\in\Omega$.
Thus we will assume that $l>r$.

Suppose that $l=1$. Then $r=0$, so that $B\le 2$
$$
L(B,D)=B-2\le 0
$$
for all $(B,D)\in\Omega$.
Moreover, in this case $L(B,D)=0$ if and only if $B=2$.
Thus we will assume that $l\ge 2$.

Suppose that $r=0$. Then $B^l\le 2$ in $\Omega$. Since $l\ge 2$, this implies that $B<1+\frac{1}{l}$.
The latter gives
$$
L(B,D)=lB-rD-(l+1)<l\cdot\left(1+\frac{1}{l}\right)-(l+1)=0.
$$
Thus we will assume that $r>0$.

The function $L$ attains its maximum
in $\Omega$ at some point $P$ by Lemma~\ref{lemma:bb2dd-attains}.
Abusing notation a little bit, we write $P=(B,D)$ and put $M=L(P)$.

If $B=1$, then
$$
M=l-rD-(l+1)=-rD-1<0.
$$
Thus we will assume that $B>1$.

Suppose that $B=D$. Then $B^{l-r}\le 2$,
which implies $B\le 1+\frac{1}{l-r}$.
This gives
$$
M=lB-rD-(l+1)=(l-r)B-(l+1)\le (l-r)\cdot\left(1+\frac{1}{l-r}\right)-(l+1)=-r<0.
$$
Thus we will assume that $D>B$.

By Theorem~\ref{theorem:Lagrange}
applied to $L$ and
$$
G_1=B^l-2D^r, \ G_2=B-D,\ G_3=1-B,
$$
one has
\begin{equation}\label{eq:bb2dd-Lagrange}
(l,-r)=
\lambda (lB^{l-1},-2rD^{r-1})
\end{equation}
for some non-negative number $\lambda$.
This implies that
$$
\lambda=\frac{1}{B^{l-1}},
$$
so that $\lambda$ is positive, and thus $B^l=2D^r$ by Theorem~\ref{theorem:Lagrange}.
Also, since $r>0$, equation~\eqref{eq:bb2dd-Lagrange} implies that
$$
1=2D^{r-1}\cdot\lambda=\frac{2D^{r-1}}{B^{l-1}}=\frac{B}{D}<1,
$$
which is a contradiction.
\end{proof}

\begin{lemma}\label{lemma:aabb2dd}
Let $k$ be a positive integer, $l$ be a non-negative integer, and
$p$ be a non-negative real number such that $k+l+p>2$; let $r$ be a non-negative integer.
Let $\Omega\subset\R^3_{A,B,D}$ be defined by inequalities
$A^{k-1}B^l\le 2D^r$, $A\ge B\ge 1$, $A\ge k+l+p$, $D\ge B$, $2A\ge D$,
and an additional inequality $D\ge A$ in the case when $k\ge 2$.
Put
$$
L(A,B,D)=(k-2)A+lB-rD+p.
$$
Then $L$ is non-positive on $\Omega$. Moreover, $L(A,B,D)$ is negative unless $A=p+2$ and~\mbox{$p>0$}.
\end{lemma}
\begin{proof}
The function $L$ attains its maximum
in $\Omega$ at some point $P$.
To see this apply Lemma~\ref{lemma:attains}
with $c=r+1$, an arbitrary $N\ge\max\{k+l,c+1\}$, and the closed subset
defined by conditions
\begin{multline*}
A_1=\ldots=A_k,\quad A_{k+1}=\ldots=A_{k+l},\quad A_{k+l+1}=\ldots=A_N=1,\\
D_1=2A_1,\quad D_2=\ldots=D_c, \quad A_1\ge k+l+p.
\end{multline*}
Abusing notation, we write $P=(A,B,D)$ and put $M=L(P)$.

If $k=1$, we have $B^l\le 2D^r$,
so that
$$
M=lB-rD+(p-A)\le lB-rD-(l+1)\le 0
$$
by Lemma~\ref{lemma:bb2dd}. Moreover, if $M=0$, then $l=1$, $r=0$, and $B=2$, so $M=p+2-A$ and $A=p+2$.
In particular, condition $k+l+p>2$ implies that $p>0$.
Thus we will assume that $k\ge 2$.

Note that $r>0$. Indeed, otherwise one has
$$
2\ge A^{k-1}B^l\ge A^{k-1}\ge A>2,
$$
which is absurd.

If $l=0$, then $A^{k-1}\le 2D^r$, so that
$$
M=(k-2)A-rD+p=((k-1)A-rD-k)+(k+p-A)< 0
$$
by Lemma~\ref{lemma:bb2dd}.
Thus we will assume that $l>0$.

If $B=1$ then $A^{k-1}\le 2D^r$, so that
$$
M=(k-1)A+l-rD-A+p=
((k-1)A-rD-k)+(k+l+p-A)< 0
$$
by Lemma~\ref{lemma:bb2dd}.
Thus we will assume that $B>1$.

If $A=B$ then $A^{k+l-1}\le 2D^r$, so that
$$
M=(k+l-2)A-rD+p=((k+l-1)A-rD-(k+l))+(k+l+p-A)< 0
$$
by Lemma~\ref{lemma:bb2dd}.
Thus we will assume that $A>B$. In particular, we have $D>B$.

Suppose that $D=A$. If $r<k-1$, then $A^{k-1-r}B^l\le 2$, so that
$A\le 2<k+l+p$, a contradiction. Hence $r\ge k-1$, and
$B^l\le 2D^{r-k+1}$. We have
$$
M=lB+(k-r-2)D+p=\left(lB-(r-k+1)D-(l+1)\right)+(l+p+1-D)<0
$$
by Lemma~\ref{lemma:bb2dd}.
Thus we will assume that $D>A$.

By Theorem~\ref{theorem:Lagrange}
applied to $L$ and
$$
G_1=A^{k-1}B^l-2D^r ,\ G_2=k+l+p-A ,\ G_3=B-A,\ G_4=1-B,\ G_5=B-D,\ G_6=A-D,
$$
one has
\begin{equation}\label{eq:aabb2dd-Lagrange}
(k-2,l,-r)=
\lambda_1 ((k-1)A^{k-2}B^l,lA^{k-1}B^{l-1},-2rD^{r-1})+\lambda_2(-1,0,0)
\end{equation}
for some non-negative numbers $\lambda_1$ and $\lambda_2$.
Since $l>0$, equation~\eqref{eq:aabb2dd-Lagrange} implies that
$$
\lambda_1=\frac{1}{A^{k-1}B^{l-1}},
$$
so that $\lambda_1$ is positive, and hence $A^{k-1}B^l=2D^r$ by Theorem~\ref{theorem:Lagrange}.
Finally, since $r>0$, equation~\eqref{eq:aabb2dd-Lagrange} implies that
$$
1=2D^{r-1}\cdot\lambda_1=\frac{2D^{r-1}}{A^{k-1}B^{l-1}}=\frac{B}{D}<1,
$$
which is a contradiction.
\end{proof}

\begin{lemma}\label{lemma:aa11ddd}
Let $k$ and $c$ be positive integers, and $l$ be a non-negative
real number such that $k+l>2$.
Let $\Omega$ be a subset of $\R^2$ with coordinates $A,D$
defined by inequalities
$A^k\le D^c$, $D\ge 2A$, and $A\ge k+l$.
Put
$$L(A,D)=kA-cD+l.$$
Then $L$ is negative on $\Omega$.
\end{lemma}
\begin{proof}
The function $L$ attains its maximum
in $\Omega$ at some point $P$.
To see this apply Lemma~\ref{lemma:attains}
with an arbitrary $N\ge\max\{k,c+1\}$, and the closed subset
defined by conditions
$$A_1=\ldots=A_k,\quad A_{k+1}=\ldots=A_N=1,\quad D_1=\ldots=D_c,\quad A_1\ge k+l.$$

Abusing notation, we write $P=(A,D)$.
Note that if $(k+1)A\le cD$, then
$$
L(A,D)=kA-cD+l\le l-A<0.
$$
In particular, this happens if $k<2c$, since $k$ and $c$ are integers.

By Theorem~\ref{theorem:Lagrange}
applied to $L$ and
$$
G_1=A^k-D^c,\ G_2=2A-D,\ G_3=k+l-A,
$$
one has
$$
(k,-c)=\lambda_1 (kA^{k-1},-cD^{c-1})+\lambda_2(2,-1)+\lambda_3(-1,0)
$$
for some non-negative numbers $\lambda_1$, $\lambda_2$, and $\lambda_3$.

Suppose that $\lambda_1>0$. By Theorem~\ref{theorem:Lagrange}
this means that $A^k=D^c$.
We can assume that~\mbox{$k\ge 2c$}. Thus $D^c\ge A^{2c}$, and $D\ge A^2$.

Assume that $c\ge 2$. Then $k+1\le ck\le c(k+l)\le cA$, so that
$$
(k+1)A\le cA^2\le cD,
$$
and we are done. Hence we have $c=1$, and
$$
L(A,D)=kA-D+l\le kA-A^2+l\le l-lA<0.
$$

Therefore, we may suppose that $\lambda_1=0$. Then
\begin{equation}\label{eq:k-c-123}
(k,-c)=(2\lambda_2-\lambda_3,-\lambda_2),
\end{equation}
so that $\lambda_2=c$ and
$D=2A$ by Theorem~\ref{theorem:Lagrange}.
Also, we see from~\eqref{eq:k-c-123} that $k\le 2c$.
If~\mbox{$k=2c$}, then
$A^{2c}=(2A)^c$, which means $A=2<k+l$, a contradiction.
Thus, we have~\mbox{$k<2c$}, which implies the assertion of the lemma.
\end{proof}

\begin{lemma}\label{lemma:aabbdd}
Let $k$ be a positive integer, let $l$ be a non-negative integer
such that $k+l>2$, and let $c$ be a non-negative integer.
Let $\Omega\subset\R^3_{A,B,D}$ be defined by inequalities
$A^kB^l\le D^c$, $D\ge 2A$, $A\ge B\ge 1$, and $A\ge k+l$.
Put
$$
L(A,B,D)=kA+lB-cD.
$$
Then $L$ is negative on $\Omega$.
\end{lemma}
\begin{proof}
The function $L$ attains its maximum
in $\Omega$ at some point $P$.
To see this apply Lemma~\ref{lemma:attains}
with an arbitrary $N\ge\max\{k+l,c+1\}$, and the closed subset
defined by conditions
\begin{multline*}
A_1=\ldots=A_k,\quad A_{k+1}=\ldots=A_{k+l},\quad A_{k+l+1}=\ldots=A_N=1,\\
D_1=\ldots=D_c, \quad A_1\ge k+l.
\end{multline*}
Abusing notation, we write $P=(A,B,D)$ and put $M=L(P)$.

Note that $c>0$. Indeed, otherwise one has
$$
1\ge A^kB^l\ge A^k\ge A>2,
$$
which is absurd.

If $l=0$, then $A^k\le D^c$, so that
$$
M=kA-cD<0
$$
by Lemma~\ref{lemma:aa11ddd}.
Thus we will assume that $l>0$.

If $A=B$, then $A^{k+l}\le D^c$, so that
$$
M=(k+l)A-cD<0
$$
by Lemma~\ref{lemma:aa11ddd}.
Thus we will assume that $A>B$.

If $B=1$, then $A^k\le D^c$, so that
$$
M=kA-cD+l<0
$$
by Lemma~\ref{lemma:aa11ddd}.
Thus we will assume that $B>1$.

If $D=2A$, then
$$
M=kA+lB-c\cdot 2A=(k-2)A+lB-(c-1)\cdot 2A<0
$$
by Lemma~\ref{lemma:aabb2dd}. Thus we will assume that $D>2A$.

By Theorem~\ref{theorem:Lagrange}
applied to $L$ and
$$
G_1=A^kB^l-D^c,\ G_2=k+l-A,\ G_3=2A-D,\ G_4=B-A,\ G_5=1-B,
$$
one has
\begin{equation}\label{eq:aabbdd-Lagrange}
(k,l,-c)=
\lambda_1 (kA^{k-1}B^l,lA^kB^{l-1},-cD^{c-1})+\lambda_2(-1,0,0)
\end{equation}
for some non-negative numbers $\lambda_1$ and $\lambda_2$.
Since $l>0$, equation~\eqref{eq:aabbdd-Lagrange} implies that
$$
\lambda_1=\frac{1}{A^kB^{l-1}},
$$
so that $\lambda_1$ is positive and $A^kB^l=D^c$ by Theorem~\ref{theorem:Lagrange}.
Finally, since $c>0$, equation~\eqref{eq:aabbdd-Lagrange} implies that
$$
1=D^{c-1}\cdot\lambda_1=\frac{D^{c-1}}{A^kB^{l-1}}=\frac{B}{D}<1,
$$
which is a contradiction.
\end{proof}

Consider a real vector space $\R^m$.
For a vector $v\in\R^m$ we denote by $v^{(i)}$ its $i$-th
coordinate.
Denote
$$
\ee_i=(\underbrace{0,\ldots,0}_{i-1},1,0,\ldots,0)\in\R^{m}, \quad
1\le i\le m.
$$

\begin{lemma}\label{lemma:go-up}
Let $u\in\R^m$ be a vector such that
$$0\le u^{(1)}\le\ldots\le u^{(m)}.$$
Put $u_{-1}=-\ee_1$, put
$u_0=\ee_1$, put $u_i=-\ee_i+\ee_{i+1}$ for $1\le i\le m-1$,
and put $u_m=-\ee_m$. Choose a subset~\mbox{$\Lambda\subset\{-1,0,\ldots,m\}$}.
Suppose that
$$\lambda u+\sum\limits_{i\in\Lambda}\lambda_i u_i=(1,\ldots,1)$$
for some non-negative number $\lambda$ and some positive
numbers $\lambda_i$.
Then there exist two indices $0\le p\le q\le m$ such that
$u^{(p+1)}=\ldots=u^{(q)}$
and
$$\{0,\ldots,p-1\}\cup\{q+1,\ldots,m\}\subset\Lambda.$$
\end{lemma}
\begin{proof}
Note that for the vector $\lambda u$ the same assumptions hold
as for the vector $u$ itself provided that $\lambda\ge 0$. Thus
we will replace $u$ by $\lambda u$ and assume that
$\lambda=1$.
In other words we have a system of equations
$$
\left\{
  \begin{array}{ll}
   u^{(1)}-\lambda_{-1}+\lambda_0-\lambda_1=1,  \\
   u^{(2)}+\lambda_1-\lambda_2=1,  \\
   \ldots   \\
   u^{(m-1)}+\lambda_{m-2}-\lambda_{m-1}=1,   \\
   u^{(m)}+\lambda_{m-1}-\lambda_m=1,
  \end{array}
\right.
$$
where we put $\lambda_i=0$ for $i\not\in\Lambda$.
Choose the indices $p$ and $q$ so that
$$u^{(p)}<u^{(p+1)}=\ldots=u^{(q)}=1<u^{(q+1)}.$$
In particular, we put $p=q$ if one has $u^{(p)}<1<u^{(p+1)}$,
we put
$p=q=0$ if $1<u^{(1)}$, and we put $p=q=m$ if $u^{(m)}<1$.

For $1<i\leqslant p$ we have
$$
1=u^{(i)}+\lambda_{i-1}-\lambda_i< 1+\lambda_{i-1},
$$
so $\lambda_{i-1}>0$.
Moreover, if $p>0$, then
$$
1=u^{(1)}-\lambda_{-1}+\lambda_0-\lambda_1< 1+\lambda_{0},
$$
so $\lambda_{0}>0$.
In the same way for $j> q$ we have
$$
1=u^{(j)}+\lambda_{j-1}-\lambda_j> 1-\lambda_j,
$$
so $\lambda_j>0$.
This exactly gives the assertion of the lemma.
\end{proof}

\begin{lemma}\label{lemma:go-down}
Let $u\in\R^m$ be a vector such that
$$0\ge u^{(1)}\ge\ldots\ge u^{(m)}.$$
Put $u_0=-\ee_1$, and put $u_i=-\ee_i+\ee_{i+1}$ for $1\le i\le m-1$.
Choose a subset~\mbox{$\Lambda\subset\{0,\ldots,m-1\}$}.
Suppose that
$$\lambda u+\sum\limits_{i\in\Lambda}\lambda_i u_i=(-1,\ldots,-1)$$
for some non-negative number $\lambda$ and some positive
numbers $\lambda_i$.
Then one of the following possibilities occurs:
\begin{itemize}
\item[(i)] $0\in\Lambda$ and $u^{(2)}=\ldots=u^{(m)}$;
\item[(ii)] $u^{(1)}=\ldots=u^{(m)}$;
\item[(iii)] $\{1,2,\ldots,m-1\}\subset\Lambda$;
\item[(iv)] $\{0,2,3,\ldots,m-1\}\subset\Lambda$.
\end{itemize}
\end{lemma}
\begin{proof}
Note that for the vector $\lambda u$ the same assumptions hold
as for the vector $u$ itself provided that $\lambda\ge 0$. Thus
we will replace $u$ by $\lambda u$ and assume that
$\lambda=1$.
In other words we have a system of equations
$$
\left\{
  \begin{array}{ll}
   u^{(1)}-\lambda_0-\lambda_1=-1,  \\
   u^{(2)}+\lambda_1-\lambda_2=-1,  \\
   \ldots   \\
   u^{(m-1)}+\lambda_{m-2}-\lambda_{m-1}=-1,   \\
   u^{(m)}+\lambda_{m-1}=-1,
  \end{array}
\right.
$$
where we put $\lambda_i=0$ for $i\not\in\Lambda$.
Suppose that $u^{(m)}=-1$.
Let $q$ be the minimal index such that $u^{(q)}=-1$.
Then, considering equations from the last one to the $q$-th
one by one we have
$$
\lambda_{m-1}=\ldots=\lambda_{q-1}=0.
$$
Moreover, if $q>2$, then
$$
u^{(q-1)}+\lambda_{q-2}-\lambda_{q-1}=u^{(q-1)}+\lambda_{q-2}>-1,
$$
which is impossible.
Thus either $q=2$, so $u^{(1)}>-1$ and $\lambda_0>0$, which corresponds to case~(i),
or $q=1$, which corresponds to case~(ii).

Now suppose that $u^{(m)}<-1$.
Choose the indices $1\le p\le q< m$ such that
$$
u^{(p)}>u^{(p+1)}=\ldots=u^{(q)}=-1>u^{(q+1)}.
$$
Then for $i>q$ one has
$$
   -1=u^{(i)}+\lambda_{i-1}-\lambda_{i}<-1+\lambda_{i-1},
$$
so $\lambda_{i-1}>0$.
Moreover, for $p\leqslant j \leqslant q$ one has
$$
   -1=u^{(j)}+\lambda_{j-1}-\lambda_{j}=-1+\lambda_{j-1}-\lambda_{j}<-1+\lambda_{j-1},
$$
so we also have $\lambda_{j-1}>0$.

If one has $\lambda_i>0$ for all $1\le i\le p-1$, then we obtain
case~(iii).
Otherwise take the maximal number $s$ with
$1\le s\le p-1$ such that $\lambda_s=0$.
If $s>1$ then
$$
   u^{(s)}+\lambda_{s-1}-\lambda_{s}>-1,
$$
which is impossible. Thus $s=1$ and $\lambda_0>0$,
which gives us case~(iv).
\end{proof}

\begin{lemma}\label{lemma:go-up-and-down}
Choose a vector
$u=\left(u^{(1)},\ldots,u^{(N+c)}\right)\in\R^{N+c}$
such that
$$0\le u^{(1)}\le \ldots u^{(N)},\quad
0\ge u^{(N+1)}\ge \ldots \ge u^{(N+c)}.$$
Put $u_{-1}=-\ee_1$ and $u_0=2\ee_1-\ee_{N+1}$. Furthermore, put
$u_i=-\ee_{i}+\ee_{i+1}$
for~\mbox{$1\le i\le N-1$} and for $N+1\le i<N+c$.
Finally, put $u_N=-\ee_N$.
Choose a subset~\mbox{$\Lambda'\subset\{-1,0,\ldots,N+c-1\}$}.
Suppose that
$$\lambda u+\sum\limits_{i\in\Lambda'}\lambda_i u_i=
(\underbrace{1,\ldots,1}_{N},\underbrace{-1,\ldots,-1}_{c})$$
for some non-negative number $\lambda$ and some positive
numbers $\lambda_i$.
Define
$$
\Lambda''=\left\{i\mid u^{(i)}=u^{(i+1)}\right\}.
$$
Then one of the following possibilities occurs:
\begin{itemize}
\item[(I)]
there is an index $0\le q\le N$ such that
$$
\{1,\ldots, q-1, q+1,\ldots, N, N+1,\ldots,N+c-1\}\subset\Lambda'\cup\Lambda''.
$$
and $N\in\Lambda'$;
\item[(II)]
there is an index
$1\le p\le N$
such that
$$
\{0,\ldots, p-1, p+1,\ldots, N-1, N+2,\ldots, N+c-1\}\subset\Lambda'\cup\Lambda''
$$
and $0\in\Lambda'$;
\item[(III)]
there are indices
$1\le p\le q\le N-1$
such that
$$
\{0,\ldots, p-1, p+1,\ldots, q-1, q+1,\ldots, N, N+2,\ldots, N+c-1\}\subset\Lambda'\cup\Lambda''
$$
and $\{0, N\}\subset\Lambda'$.
\end{itemize}
\end{lemma}
\begin{proof}
Apply Lemma~\ref{lemma:go-up} to the first $N$ coordinates of
the vector $u$ and Lemma~\ref{lemma:go-down}
to its last $c$ coordinates.
Define numbers $p$ and $q$ following the notation of
Lemma~\ref{lemma:go-up}. The only possibility
to have $0\not\in\Lambda$ is to have $p=0$.
Then $\{1,\ldots,q-1\}\subset \Lambda''$ and~\mbox{$\{q+1,\ldots,N\}\subset \Lambda'$}.
Moreover, for the last $c$ coordinates only cases (ii) or (iii) from Lemma~\ref{lemma:go-down} can occur,
which gives us case~(I).

So we can assume that $0\in \Lambda'$.
From Lemma~\ref{lemma:go-down} one can easily see that
$$
\{N+2,\ldots,N+c-1\}\subset \Lambda'\cup \Lambda''.
$$
If $q=N$ then $\{0,\ldots, p-1\}\subset\Lambda'$ and
$\{p+1,\ldots, N-1\}\subset\Lambda''$, and we obtain case~(II).
If $q<N$ then
$$
\{0,\ldots, p-1, q+1,\ldots,N\}\subset\Lambda'
$$
and
$\{p+1,\ldots, q-1\}\subset\Lambda''$, and we obtain case~(III).
\end{proof}

Now we are ready to prove the main result of this section.

\begin{proposition}\label{proposition:Lagrange}
Let $c$ and $N$ be positive integers such that
$N\ge 2c+1$.
Let $\Omega$ be a subset of $\R^{N+c}$ with coordinates
$A_1, \ldots, A_N, D_1, \ldots, D_c$
defined by inequalities
\begin{gather*}
A_1\cdot\ldots\cdot A_N\le D_1\cdot\ldots\cdot D_c,\\
A_1\ge\ldots\ge A_N\ge 1, \quad D_1\ge\ldots\ge D_c,\\
D_1\ge 2A_1, \ D_2\ge A_2,\ \ldots, \ D_c\ge A_c.
\end{gather*}
Put
$$
L(A_1,\ldots,A_N,D_1,\ldots,D_c)=A_1+\ldots+A_N-D_1-\ldots-D_c.
$$
Then $L$ is non-positive on the subset of $\Omega$ where $A_1\ge N$.
\end{proposition}
\begin{proof}
Rewrite the inequalities defining $\Omega$ as
\begin{gather*}
A_1\cdot\ldots\cdot A_N-D_1\cdot\ldots\cdot D_c\le 0,\\
A_2-A_1\le 0,\ \ldots, \ A_N-A_{N-1}\le 0,\ 1-A_N\le 0,\\
D_2-D_1\le 0,\ \ldots, \ D_c-D_{c-1}\le 0, \\
A_2-D_2\le 0, \ \ldots, \ A_c-D_c\le 0,\ 2A_1-D_1\le 0.
\end{gather*}

The function $L$ attains its maximum
in $\Omega$ at some point $P\in\Omega$ by Lemma~\ref{lemma:attains}.
Abusing notation a little bit, we write~\mbox{$P=(A_1,\ldots,A_N,D_1,\ldots,D_c)$}.
If for some $2\le i\le c$ one has~\mbox{$A_i=D_i$}, we
cancel $A_i$ and $D_i$ from the inequalities defining $\Omega$ and
from the definition of~$L$ and arrive to the same assertion with
a smaller number of parameters.
Therefore, we assume that for all $2\le i\le c$ one has
$A_i<D_i$ (in particular, this is the case when~\mbox{$c=1$}).
Note that after such cancelation the
condition $N\ge 2c+1$ is preserved.

Denote $\Pi_A=A_1\cdot\ldots\cdot A_N$ and $\Pi_D=D_1\cdot\ldots\cdot D_c$.
Applying Theorem~\ref{theorem:Lagrange} and keeping in
mind that
$A_i<D_i$ for $2\le i\le c$, we obtain an equality
\begin{multline*}
(\underbrace{1,\ldots,1}_{N},\underbrace{-1,\ldots,-1}_{c})=
\lambda\left(\frac{\Pi_A}{A_1},\ldots,\frac{\Pi_A}{A_N},
\frac{\Pi_D}{D_1},\ldots,\frac{\Pi_D}{D_c}\right)+\\+
\sum_{i=1}^{N-1} \lambda_i (\underbrace{0,\ldots,0}_{i-1},-1,1,0,\ldots,0)
+\lambda_N (\underbrace{0,\ldots,0}_{N-1},-1,0,\ldots,0)+\\
+\sum_{i=N+1}^{N+c-1}
\lambda_i (\underbrace{0,\ldots,0}_{i-1},-1,1,0,\ldots,0)+
\lambda_0(2,\underbrace{0,\ldots,0}_{N-1},-1,0,\ldots,0)+
\lambda_{-1}(-1,0,\ldots,0)
\end{multline*}
for some non-negative numbers $\lambda$ and $\lambda_i$, where $-1\le i\le N+c-1$.

Let $\Lambda'\subset\{-1,0,1\ldots,N+c-1\}$ be the set of indices
such that for any $i\in\Lambda'$ one has~\mbox{$\lambda_i>0$},
and let $\Lambda''=\{i\mid u^{(i)}=u^{(i+1)}\}$.
By
Theorem~\ref{theorem:Lagrange} for any~\mbox{$i\in\Lambda'$} the corresponding
inequality turns into equality. Thus, by
Lemma~\ref{lemma:go-up-and-down} we have the following possibilities:
\begin{itemize}
\item[(I)]
$$(A_1,\ldots,A_N,D_1,\ldots,D_c)=
(\underbrace{\overbrace{A,\ldots,A}^k,\overbrace{B,\ldots,B}^l}_N,\underbrace{D,\ldots,D}_c)$$
for some $k\ge 1$ and $l\ge 0$;
\item[(II)]
$$(A_1,\ldots,A_N,D_1,\ldots,D_c)=
(\underbrace{\overbrace{A,\ldots,A}^k,\overbrace{B,\ldots,B}^l}_N,2A,\underbrace{D,\ldots,D}_{c-1})$$
for some $k\ge 1$ and $l\ge 0$;
\item[(III)]
$$(A_1,\ldots,A_N,D_1,\ldots,D_c)=
(\underbrace{\overbrace{A,\ldots,A}^k,\overbrace{B,\ldots,B}^l,\overbrace{1,\ldots,1}^p}_N,
2A,\underbrace{D,\ldots,D}_{c-1})$$
for some $k\ge 1$, $l\ge 1$, and $p\ge 1$.
\end{itemize}
In all cases one has $c\ge 1$ and $N\ge 2c+1\ge 3$.
In cases~(II) and~(III) the inequality~\mbox{$2A\ge D$} holds.

In case~(I) we have $L(P)<0$ provided that $A_1\ge N$; to see this apply
Lemma~\ref{lemma:aabbdd}.
In case~(II) we have $L(P)<0$ provided that $A_1\ge N$; to see this apply
Lemma~\ref{lemma:aabb2dd} with~\mbox{$r=c-1$} and $p=0$.
Finally, in case~(III) we have $L(P)\le 0$ provided that~\mbox{$A_1\ge N$};
to see this apply Lemma~\ref{lemma:aabb2dd} with $r=c-1$.
\end{proof}

\end{document}